\documentclass[12pt]{article}

\usepackage{amssymb,amsmath,amsthm}
\usepackage{amsfonts}
\usepackage{latexsym}
\usepackage{amsfonts}
\usepackage{xcolor}
\usepackage{tikz}
\usepackage{mathtools}
\usepackage{mathrsfs}
\usepackage{hyperref}
\usepackage{caption}
\usepackage{subcaption}
\usepackage{adjustbox}
\usepackage{gensymb}
\usepackage{float}
\usetikzlibrary{arrows.meta}
\usetikzlibrary{plotmarks}

\usepackage[algo2e,ruled,vlined,]{algorithm2e}
\usepackage{algorithmic}

\usepackage{graphicx} 
\graphicspath{ {./images/} }

\usepackage{fullpage}

\newtheorem{theorem}{Theorem}[section] 

\newtheorem{lemma}[theorem]{Lemma}

\newtheorem{observation}[theorem]{Observation}

\newcommand{\dist}{\mathrm{dist}}

\title{Firefighting on the Hexagonal Grid and on Infinite Trees}

\author{Abdullah Dean\thanks{University of Illinois at Urbana-Champaign, email: \texttt{anoora2@Illinois.edu}} \and
Sean English\thanks{University of Illinois at Urbana-Champaign, email: \texttt{senglish@Illinois.edu}} \and
Tongyun Huang\thanks{University of Illinois at Urbana-Champaign, email: \texttt{tongyun2@Illinois.edu}} \and
Robert A. Krueger\thanks{University of Illinois at Urbana-Champaign, email: \texttt{rak5@Illinois.edu}} \and
Andy Lee\thanks{University of Illinois at Urbana-Champaign, email: \texttt{andy2@Illinois.edu}} \and
Mose Mizrahi\thanks{University of Illinois at Urbana-Champaign, email: \texttt{mosem2@Illinois.edu}} \and
Casey Wheaton-Werle\thanks{University of Illinois at Urbana-Champaign, email: \texttt{caseydw2@Illinois.edu}}}

\date{}

\begin{document}

\maketitle

\begin{abstract}
    The firefighter problem with $k$ firefighters on an infinite graph $G$ is an iterative graph process, defined as follows: Suppose a fire breaks out at a given vertex $v\in V(G)$ on Turn 1. On each subsequent even turn, $k$ firefighters protect $k$ vertices that are not on fire, and on each subsequent odd turn, any vertex that is on fire spreads the fire to all adjacent unprotected vertices. The firefighters' goal is to eventually stop the spread of the fire. If there exists a strategy for $k$ firefighters to eventually stop the spread of the fire, then we say $G$ is $k$-containable.
    
    We consider the firefighter problem on the \emph{hexagonal grid}, which is the graph whose vertices and edges are exactly the vertices and edges of a regular hexagonal tiling of the plane. It is not known if the hexagonal grid is $1$-containable. In [T.~Gaven\v{c}iak, J.~Kratochv\'{\i}l and P.~Pra{\l}at. Firefighting on square, hexagonal, and triangular grids. \textit{Discrete Math.}, 337:142-155, 2014.], it was shown that if the firefighters have one firefighter per turn and one extra firefighter on two turns, the firefighters can contain the fire. We improve on this result by showing that even with only one extra firefighter on one turn, the firefighters can still contain the fire.
    
    In addition, we explore $k$-containability for \emph{birth sequence trees}, which are rooted trees having the property that every vertex at the same level has the same degree. A \emph{birth sequence forest} is a forest, each component of which is a birth sequence tree. For birth sequence trees and forests, the fire always starts at the root of each tree. We provide a pseudopolynomial time algorithm to decide if all the vertices at a fixed level can be protected or not.
\end{abstract}

\section{Introduction}
The problem of firefighters on graphs studies the following iterative process: given a graph $G$, a subset of the vertices are initially on fire on turn $1$. Then in alternating turns some vertices are protected by firefighters and the fire spreads to all unprotected vertices adjacent to a vertex on fire. Once a vertex has been protected by a firefighter it is protected for the remainder of the process. Similarly once a vertex is on fire it remains on fire.

More formally, let $V$ be the set of vertices, let $F^{(t)}$ be the set of vertices on fire on turn $t$, and let $P^{(t)}$ be the set of protected vertices on turn $t$. Initially, on turn $1$, $F^{(1)}$ is some non-empty subset of $V$, and $P^{(1)}$ is empty. On an even turn $2t$, we let $P^{(2t)}$ be the union of $P^{(2t-1)}$ and a subset of $V \setminus F^{(2t-1)}$ and let $F^{(2t)} = F^{(2t-1)}$. On an odd turn $2t + 1$, where $t \geq 1$, we let $F^{(2t+1)}$ be $N[F^{(2t)}] \setminus P^{(2t)}$, the closed neighborhood of $F^{(2t)}$ except for the vertices in $P^{(2t)}$, and let $P^{(2t + 1)} = P^{(2t)}$. For simplicity when we say a firefighter protects a vertex on turn $t$ we assume this $t$ to be even.

Let $t \in \mathbb{N}$. On turn $t$, we say that a vertex is \textbf{on fire} or \textbf{burning} if it is in $F^{(t)}$, \textbf{protected} if it is in $P^{(t)}$, and \textbf{unprotected} if it is in neither of these sets. We say a vertex is \textbf{saved} if it is impossible for it to ever be on fire. That is, a vertex is saved if it is in $P^{(t)}$ or in a component of the subgraph induced by $V \setminus P^{(t)}$ with no burning vertices.

When firefighting on an infinite graph, we say the fire is \textbf{contained} if all but finitely many vertices are saved. We say an infinite graph given with a subset of vertices initially on fire is \textbf{$k$-containable} if the fire can be contained by protecting at most $k$ vertices every even turn.

Problems related to $k$-containability also exist on finite graphs. For example, there is the NP-Complete decision problem of whether it is possible to save all vertices in a set $S$ by protecting at most $k$ vertices every even turn \cite{NPCCubic}.

Firefighting on graphs can be used to model network spread, and can be used to understand the spread of computer viruses, misinformation, and infectious diseases. Indeed, similar problems arise in SIR epidemic models where a disease seeded at an initial set of vertices spreads through a network. Effects of vaccination programs in such models where vertices are granted immunity from infection have been extensively studied \cite{EpiSurvey}.
Conversely, the firefighters can also be thought of as an adversary; for example, the fires and the firefighters could model a broadcast signal and an adversary trying to censor it in a communication network. In this context, it would be good for the communication network to be robust against censorship. There is a significant body of existing related work, some of which is discussed in \cite{FireSurvey}.

\subsection{The Hexagonal Grid}
The problem of $k$-containability has been studied on various infinite graphs. The infinite triangular grid, formed by tiling the plane with equilateral triangles and letting the corners be vertices, with a single initially burning vertex is conjectured to not be $2$-containable \cite{Grids}. Here, we focus on the infinite hexagonal grid, formed by tiling the plane with equilateral hexagons with the corners as vertices. It is conjectured that the hexagonal grid is not $1$-containable, \cite{hexgridconj}.

It is known that all orientations of the hexagonal grid are $1$-containable as in a directed graph fire can only spread to out-neighbors, \cite{SlideThm}. The following theorem suggests that if the hexagonal grid is not $1$-containable, then it is ``barely'' not $1$-containable.

\begin{theorem}[\cite{Grids}]\label{theorem 2extrafirefighters}
If it is possible to use an additional firefighter at two turns, $2t_1$ and $2t_2$ possibly with $2t_1 = 2t_2$, then one firefighter every turn is sufficient to contain the fire on the hexagonal grid with a single initially burning vertex.
\end{theorem}

Our main contribution is an improvement to this result.

\begin{theorem}\label{theorem main theorem}
If it is possible to use an additional firefighter at a single turn, $2\tau$, then one firefighter every turn is sufficient to contain the fire on the hexagonal grid with a single initially burning vertex.
\end{theorem}
Our theorem shows that the hexagonal grid conjecture, if true, is in some sense sharp; it would not be true if even a single extra firefighter was available. Our firefighting strategy, like the two-extra-firefighters strategy in \cite{Grids}, does not need to know in advance which turn the extra firefighters can be used.

\subsection{Birth Sequence Trees}
We also study the firefighter problem on rooted trees with their roots on fire. On a rooted tree, the \textbf{depth} of a vertex refers to its distance from the root. The root has depth $0$. Such trees have properties that simplify the protection of vertices. Notably, it is always optimal to use \textbf{hot} strategies on such trees, which are strategies where only vertices with burning neighbors are protected.

It is known that on a finite tree with its root on fire, checking whether all vertices in a set $S$ can be saved with $k$ firefighters every turn is NP-Complete when the maximum degree of the tree is at least $k + 2$. \cite{TreesNPC}

We study a more restricted class of trees: \textbf{birth sequence} trees. These are rooted trees characterized by a \textbf{birth sequence} $d_0,d_1,d_2,\dots$ such that vertices at depth $k$ have $d_k$ children. These trees have the property that for any $k$, if all vertices at depth less than $k$ are removed, then all component trees in the resulting forest are isomorphic. For an example of a birth sequence tree, consider the infinite binary tree whose birth sequence is $d_k = 2$ for all $k \in \mathbb{Z}_{\geq0}$.

Our first result is a necessary and sufficient condition for the $k$-containability of an infinite birth sequence tree with its root on fire: 

\begin{theorem}\label{theorem infinite trees}
Let $T$ be an infinite rooted tree with birth sequence $d_0,d_1,d_2,\dots$, with its root initially on fire. Then $T$ is $k$-containable if and only if there exists some $t \in \mathbb{Z}_{\geq0}$ such that 
\[
\prod_{i=0}^td_i - k\bigg(1+\sum_{i=1}^t\prod_{j=i}^td_j\bigg)\leq 0.
\]
\end{theorem}

We also provide a generalization of this theorem for forests of infinite birth sequence trees, each with its root on fire. Informed by this generalization, we provide an algorithmic result as well. For forests of birth sequence trees with burning roots containing $m$ trees with height $n$, we give a dynamic programming algorithm that can determine in $O(m\binom{m + k - 1}{m - 1}n(kn)^m)$ steps whether it is possible to save all leaves with $k$ firefighters every turn.

\subsection{Further Terminology}
\textbf{Vulnerable} vertices are unprotected vertices that have neighbors on fire. \textbf{Actively burning} vertices are burning vertices with vulnerable neighbors. Every odd turn after turn $1$, fire spreads from the actively burning vertices to the vulnerable vertices. We will call a strategy \textbf{hot} if we only protect vulnerable vertices (assuming there are at least as many vulnerable vertices as vertices we can protect this round).

\section{The Hexagonal Grid}\label{section hex grid}

In this section, we prove Theorem \ref{theorem main theorem}. Our strategy is very similar to the one given in \cite{Grids} to prove Theorem \ref{theorem 2extrafirefighters}, however we optimize the strategy at certain points to contain the fire without a second extra firefighter.

For the strategy description and proof, we fix some notation. Let $V$ be the set of vertices of the hexagonal grid. Let $\tau^* \in \mathbb{Z}_{>0}$ be such that at turn $2\tau^*$, two firefighters can be used. To simplify the proof, we wish to use our extra firefighter on turn $2\tau$, where 
\[
\tau=\begin{cases} \tau^*&\text{ if }\tau\text{ is odd,}\\
\tau^*+1&\text{ if }\tau^*\text{ is even.}
\end{cases}
\]
In this way, we can enforce that $\tau$ is odd, and if $\tau^*$ is even, we will play our extra firefighter in the place we would if we were given the firefighter on turn $2(\tau^*+1)=2\tau$ instead of turn $2\tau^*$.

Parts of our strategy are essentially the same as the strategy given in \cite{Grids}, but for completeness we provide all the details here. To be able to address vertices of the hexagonal grid, we draw the hexagonal grid on the Cartesian plane with regular hexagons, here the initial vertex on fire, $f$, is at the origin, and the grid is oriented such that there is a vertex adjacent to $f$ directly above it, and that every edge of the graph has length $1$. Note that throughout this section, every reference to distance will be distance in the hexagonal grid, not Euclidean distance. To be able to address vertices without using square roots or fractions, we make a change of coordinates: Let $i = \frac{2}{\sqrt{3}}x$, and let $j = 2y$. The vertex $(i,j)$ corresponds to the point $(x,y) = (\frac{i\sqrt{3}}{2}, \frac{j}{2})$ on the Cartesian plane. See Figure \ref{Coordinate System} for this mapping of the hexagonal grid on the Cartesian plane. Note that $(i,j)$ is a vertex of the hexagonal grid if and only if $(i\bmod{2}, j\bmod{6}) \in \{(0,0),(0,2),(1,3),(1,5)\}$.

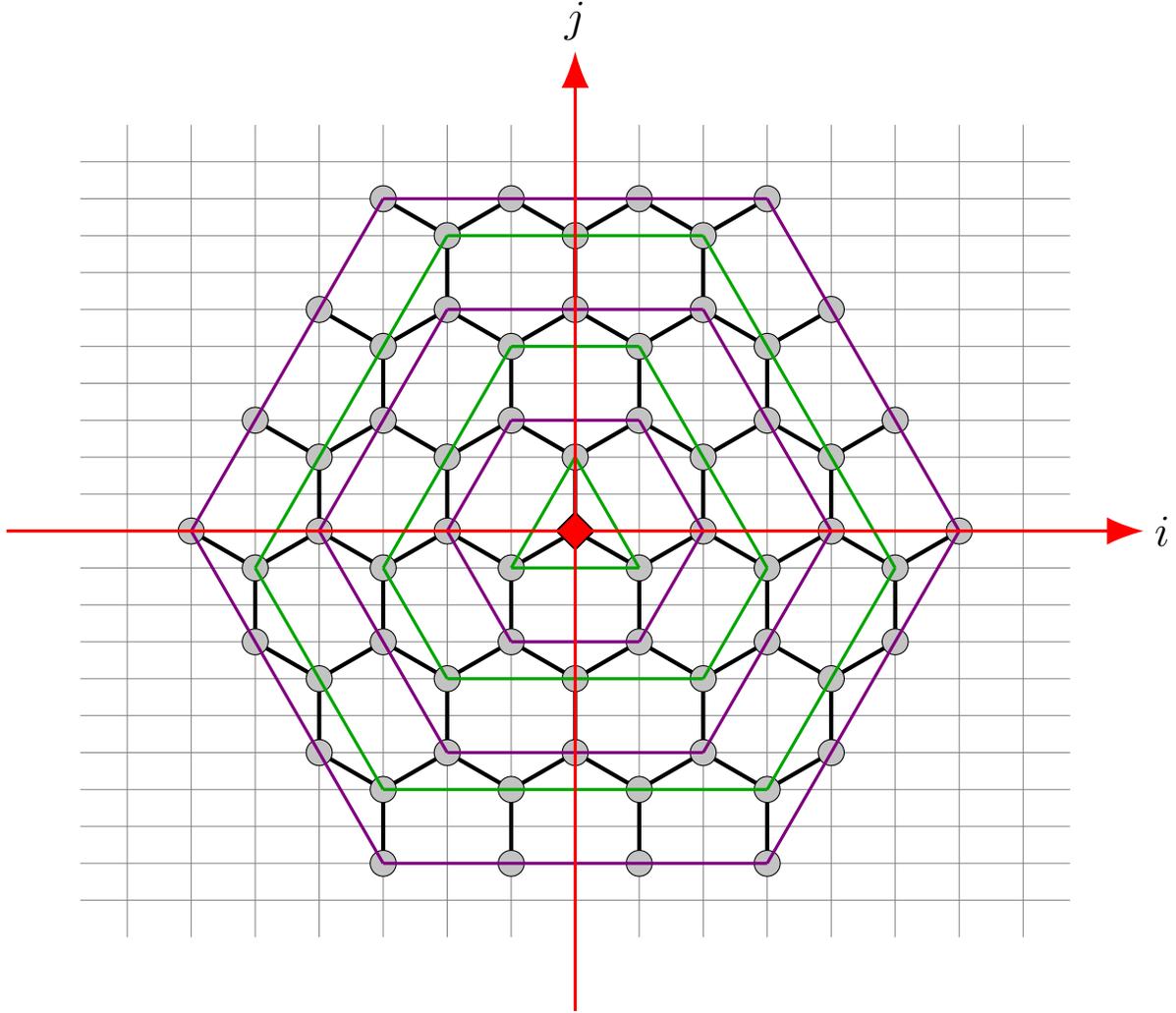
\begin{figure}[!ht]
\centering
\begin{tikzpicture}
\definecolor{lightgray}{RGB}{195,195,195}
\definecolor{darkgreen}{RGB}{0,165,0}
\draw [very thin,gray] (-6.696152422706632,-5.0)--(6.696152422706632,-5.0);
\draw [very thin,gray] (-6.696152422706632,-4.5)--(6.696152422706632,-4.5);
\draw [very thin,gray] (-6.696152422706632,-4.0)--(6.696152422706632,-4.0);
\draw [very thin,gray] (-6.696152422706632,-3.5)--(6.696152422706632,-3.5);
\draw [very thin,gray] (-6.696152422706632,-3.0)--(6.696152422706632,-3.0);
\draw [very thin,gray] (-6.696152422706632,-2.5)--(6.696152422706632,-2.5);
\draw [very thin,gray] (-6.696152422706632,-2.0)--(6.696152422706632,-2.0);
\draw [very thin,gray] (-6.696152422706632,-1.5)--(6.696152422706632,-1.5);
\draw [very thin,gray] (-6.696152422706632,-1.0)--(6.696152422706632,-1.0);
\draw [very thin,gray] (-6.696152422706632,-0.5)--(6.696152422706632,-0.5);
\draw [very thin,gray] (-6.696152422706632,0.5)--(6.696152422706632,0.5);
\draw [very thin,gray] (-6.696152422706632,1.0)--(6.696152422706632,1.0);
\draw [very thin,gray] (-6.696152422706632,1.5)--(6.696152422706632,1.5);
\draw [very thin,gray] (-6.696152422706632,2.0)--(6.696152422706632,2.0);
\draw [very thin,gray] (-6.696152422706632,2.5)--(6.696152422706632,2.5);
\draw [very thin,gray] (-6.696152422706632,3.0)--(6.696152422706632,3.0);
\draw [very thin,gray] (-6.696152422706632,3.5)--(6.696152422706632,3.5);
\draw [very thin,gray] (-6.696152422706632,4.0)--(6.696152422706632,4.0);
\draw [very thin,gray] (-6.696152422706632,4.5)--(6.696152422706632,4.5);
\draw [very thin,gray] (-6.696152422706632,5.0)--(6.696152422706632,5.0);
\draw [very thin,gray] (-6.06217782649107,-5.5)--(-6.06217782649107,5.5);
\draw [very thin,gray] (-5.196152422706632,-5.5)--(-5.196152422706632,5.5);
\draw [very thin,gray] (-4.330127018922193,-5.5)--(-4.330127018922193,5.5);
\draw [very thin,gray] (-3.4641016151377544,-5.5)--(-3.4641016151377544,5.5);
\draw [very thin,gray] (-2.598076211353316,-5.5)--(-2.598076211353316,5.5);
\draw [very thin,gray] (-1.7320508075688772,-5.5)--(-1.7320508075688772,5.5);
\draw [very thin,gray] (-0.8660254037844386,-5.5)--(-0.8660254037844386,5.5);
\draw [very thin,gray] (0.8660254037844386,-5.5)--(0.8660254037844386,5.5);
\draw [very thin,gray] (1.7320508075688772,-5.5)--(1.7320508075688772,5.5);
\draw [very thin,gray] (2.598076211353316,-5.5)--(2.598076211353316,5.5);
\draw [very thin,gray] (3.4641016151377544,-5.5)--(3.4641016151377544,5.5);
\draw [very thin,gray] (4.330127018922193,-5.5)--(4.330127018922193,5.5);
\draw [very thin,gray] (5.196152422706632,-5.5)--(5.196152422706632,5.5);
\draw [very thin,gray] (6.06217782649107,-5.5)--(6.06217782649107,5.5);
\draw [ultra thick] (-0.8660254, -0.5)--(0, 0);
\draw [ultra thick] (0.8660254, -0.5)--(0, 0);
\draw [ultra thick] (0, 1)--(0, 0);
\draw [ultra thick] (-1.732051, 0)--(-0.8660254, -0.5);
\draw [ultra thick] (-0.8660254, -1.5)--(-0.8660254, -0.5);
\draw [ultra thick] (1.732051, 0)--(0.8660254, -0.5);
\draw [ultra thick] (0.8660254, -1.5)--(0.8660254, -0.5);
\draw [ultra thick] (-0.8660254, 1.5)--(0, 1);
\draw [ultra thick] (0.8660254, 1.5)--(0, 1);
\draw [ultra thick] (-2.598076, -0.5)--(-1.732051, 0);
\draw [ultra thick] (-1.732051, 1)--(-0.8660254, 1.5);
\draw [ultra thick] (-1.732051, 1)--(-1.732051, 0);
\draw [ultra thick] (-1.732051, -2)--(-0.8660254, -1.5);
\draw [ultra thick] (0, -2)--(-0.8660254, -1.5);
\draw [ultra thick] (0, -2)--(0.8660254, -1.5);
\draw [ultra thick] (2.598076, -0.5)--(1.732051, 0);
\draw [ultra thick] (1.732051, 1)--(0.8660254, 1.5);
\draw [ultra thick] (1.732051, 1)--(1.732051, 0);
\draw [ultra thick] (1.732051, -2)--(0.8660254, -1.5);
\draw [ultra thick] (-0.8660254, 2.5)--(-0.8660254, 1.5);
\draw [ultra thick] (0.8660254, 2.5)--(0.8660254, 1.5);
\draw [ultra thick] (-3.464102, 0)--(-2.598076, -0.5);
\draw [ultra thick] (-2.598076, -1.5)--(-1.732051, -2);
\draw [ultra thick] (-2.598076, -1.5)--(-2.598076, -0.5);
\draw [ultra thick] (-2.598076, 1.5)--(-1.732051, 1);
\draw [ultra thick] (-1.732051, -3)--(-1.732051, -2);
\draw [ultra thick] (0, -3)--(0, -2);
\draw [ultra thick] (3.464102, 0)--(2.598076, -0.5);
\draw [ultra thick] (2.598076, -1.5)--(1.732051, -2);
\draw [ultra thick] (2.598076, -1.5)--(2.598076, -0.5);
\draw [ultra thick] (2.598076, 1.5)--(1.732051, 1);
\draw [ultra thick] (1.732051, -3)--(1.732051, -2);
\draw [ultra thick] (-1.732051, 3)--(-0.8660254, 2.5);
\draw [ultra thick] (0, 3)--(-0.8660254, 2.5);
\draw [ultra thick] (0, 3)--(0.8660254, 2.5);
\draw [ultra thick] (1.732051, 3)--(0.8660254, 2.5);
\draw [ultra thick] (-4.330127, -0.5)--(-3.464102, 0);
\draw [ultra thick] (-3.464102, 1)--(-2.598076, 1.5);
\draw [ultra thick] (-3.464102, 1)--(-3.464102, 0);
\draw [ultra thick] (-3.464102, -2)--(-2.598076, -1.5);
\draw [ultra thick] (-2.598076, 2.5)--(-1.732051, 3);
\draw [ultra thick] (-2.598076, 2.5)--(-2.598076, 1.5);
\draw [ultra thick] (-2.598076, -3.5)--(-1.732051, -3);
\draw [ultra thick] (-0.8660254, -3.5)--(-1.732051, -3);
\draw [ultra thick] (-0.8660254, -3.5)--(0, -3);
\draw [ultra thick] (0.8660254, -3.5)--(0, -3);
\draw [ultra thick] (0.8660254, -3.5)--(1.732051, -3);
\draw [ultra thick] (4.330127, -0.5)--(3.464102, 0);
\draw [ultra thick] (3.464102, 1)--(2.598076, 1.5);
\draw [ultra thick] (3.464102, 1)--(3.464102, 0);
\draw [ultra thick] (3.464102, -2)--(2.598076, -1.5);
\draw [ultra thick] (2.598076, 2.5)--(1.732051, 3);
\draw [ultra thick] (2.598076, 2.5)--(2.598076, 1.5);
\draw [ultra thick] (2.598076, -3.5)--(1.732051, -3);
\draw [ultra thick] (-1.732051, 4)--(-1.732051, 3);
\draw [ultra thick] (0, 4)--(0, 3);
\draw [ultra thick] (1.732051, 4)--(1.732051, 3);
\draw [ultra thick] (-5.196152, 0)--(-4.330127, -0.5);
\draw [ultra thick] (-4.330127, -1.5)--(-3.464102, -2);
\draw [ultra thick] (-4.330127, -1.5)--(-4.330127, -0.5);
\draw [ultra thick] (-4.330127, 1.5)--(-3.464102, 1);
\draw [ultra thick] (-3.464102, -3)--(-2.598076, -3.5);
\draw [ultra thick] (-3.464102, -3)--(-3.464102, -2);
\draw [ultra thick] (-3.464102, 3)--(-2.598076, 2.5);
\draw [ultra thick] (-2.598076, -4.5)--(-2.598076, -3.5);
\draw [ultra thick] (-0.8660254, -4.5)--(-0.8660254, -3.5);
\draw [ultra thick] (0.8660254, -4.5)--(0.8660254, -3.5);
\draw [ultra thick] (5.196152, 0)--(4.330127, -0.5);
\draw [ultra thick] (4.330127, -1.5)--(3.464102, -2);
\draw [ultra thick] (4.330127, -1.5)--(4.330127, -0.5);
\draw [ultra thick] (4.330127, 1.5)--(3.464102, 1);
\draw [ultra thick] (3.464102, -3)--(2.598076, -3.5);
\draw [ultra thick] (3.464102, -3)--(3.464102, -2);
\draw [ultra thick] (3.464102, 3)--(2.598076, 2.5);
\draw [ultra thick] (2.598076, -4.5)--(2.598076, -3.5);
\draw [ultra thick] (-2.598076, 4.5)--(-1.732051, 4);
\draw [ultra thick] (-0.8660254, 4.5)--(-1.732051, 4);
\draw [ultra thick] (-0.8660254, 4.5)--(0, 4);
\draw [ultra thick] (0.8660254, 4.5)--(0, 4);
\draw [ultra thick] (0.8660254, 4.5)--(1.732051, 4);
\draw [ultra thick] (2.598076, 4.5)--(1.732051, 4);
\draw [black,fill=red] plot[rotate=45,mark=square*,mark size=5pt] (0,0);
\draw [fill=lightgray] (-0.8660254, -0.5) circle (5pt);
\draw [fill=lightgray] (0.8660254, -0.5) circle (5pt);
\draw [fill=lightgray] (0, 1) circle (5pt);
\draw [fill=lightgray] (-1.732051, 0) circle (5pt);
\draw [fill=lightgray] (-0.8660254, -1.5) circle (5pt);
\draw [fill=lightgray] (1.732051, 0) circle (5pt);
\draw [fill=lightgray] (0.8660254, -1.5) circle (5pt);
\draw [fill=lightgray] (-0.8660254, 1.5) circle (5pt);
\draw [fill=lightgray] (0.8660254, 1.5) circle (5pt);
\draw [fill=lightgray] (-2.598076, -0.5) circle (5pt);
\draw [fill=lightgray] (-1.732051, 1) circle (5pt);
\draw [fill=lightgray] (-1.732051, -2) circle (5pt);
\draw [fill=lightgray] (0, -2) circle (5pt);
\draw [fill=lightgray] (2.598076, -0.5) circle (5pt);
\draw [fill=lightgray] (1.732051, 1) circle (5pt);
\draw [fill=lightgray] (1.732051, -2) circle (5pt);
\draw [fill=lightgray] (-0.8660254, 2.5) circle (5pt);
\draw [fill=lightgray] (0.8660254, 2.5) circle (5pt);
\draw [fill=lightgray] (-3.464102, 0) circle (5pt);
\draw [fill=lightgray] (-2.598076, -1.5) circle (5pt);
\draw [fill=lightgray] (-2.598076, 1.5) circle (5pt);
\draw [fill=lightgray] (-1.732051, -3) circle (5pt);
\draw [fill=lightgray] (0, -3) circle (5pt);
\draw [fill=lightgray] (3.464102, 0) circle (5pt);
\draw [fill=lightgray] (2.598076, -1.5) circle (5pt);
\draw [fill=lightgray] (2.598076, 1.5) circle (5pt);
\draw [fill=lightgray] (1.732051, -3) circle (5pt);
\draw [fill=lightgray] (-1.732051, 3) circle (5pt);
\draw [fill=lightgray] (0, 3) circle (5pt);
\draw [fill=lightgray] (1.732051, 3) circle (5pt);
\draw [fill=lightgray] (-4.330127, -0.5) circle (5pt);
\draw [fill=lightgray] (-3.464102, 1) circle (5pt);
\draw [fill=lightgray] (-3.464102, -2) circle (5pt);
\draw [fill=lightgray] (-2.598076, 2.5) circle (5pt);
\draw [fill=lightgray] (-2.598076, -3.5) circle (5pt);
\draw [fill=lightgray] (-0.8660254, -3.5) circle (5pt);
\draw [fill=lightgray] (0.8660254, -3.5) circle (5pt);
\draw [fill=lightgray] (4.330127, -0.5) circle (5pt);
\draw [fill=lightgray] (3.464102, 1) circle (5pt);
\draw [fill=lightgray] (3.464102, -2) circle (5pt);
\draw [fill=lightgray] (2.598076, 2.5) circle (5pt);
\draw [fill=lightgray] (2.598076, -3.5) circle (5pt);
\draw [fill=lightgray] (-1.732051, 4) circle (5pt);
\draw [fill=lightgray] (0, 4) circle (5pt);
\draw [fill=lightgray] (1.732051, 4) circle (5pt);
\draw [fill=lightgray] (-5.196152, 0) circle (5pt);
\draw [fill=lightgray] (-4.330127, -1.5) circle (5pt);
\draw [fill=lightgray] (-4.330127, 1.5) circle (5pt);
\draw [fill=lightgray] (-3.464102, -3) circle (5pt);
\draw [fill=lightgray] (-3.464102, 3) circle (5pt);
\draw [fill=lightgray] (-2.598076, -4.5) circle (5pt);
\draw [fill=lightgray] (-0.8660254, -4.5) circle (5pt);
\draw [fill=lightgray] (0.8660254, -4.5) circle (5pt);
\draw [fill=lightgray] (5.196152, 0) circle (5pt);
\draw [fill=lightgray] (4.330127, -1.5) circle (5pt);
\draw [fill=lightgray] (4.330127, 1.5) circle (5pt);
\draw [fill=lightgray] (3.464102, -3) circle (5pt);
\draw [fill=lightgray] (3.464102, 3) circle (5pt);
\draw [fill=lightgray] (2.598076, -4.5) circle (5pt);
\draw [fill=lightgray] (-2.598076, 4.5) circle (5pt);
\draw [fill=lightgray] (-0.8660254, 4.5) circle (5pt);
\draw [fill=lightgray] (0.8660254, 4.5) circle (5pt);
\draw [fill=lightgray] (2.598076, 4.5) circle (5pt);
\draw [darkgreen,very thick](0.0, 1.0)--(0.8660254037844386, -0.5);
\draw [darkgreen,very thick](0.8660254037844386, -0.5)--(-0.8660254037844386, -0.5);
\draw [darkgreen,very thick](-0.8660254037844386, -0.5)--(0.0, 1.0);
\draw [violet,very thick](-0.8660254037844386, 1.5)--(0.8660254037844386, 1.5);
\draw [violet,very thick](0.8660254037844386, 1.5)--(1.7320508075688772, 0.0);
\draw [violet,very thick](1.7320508075688772, 0.0)--(0.8660254037844386, -1.5);
\draw [violet,very thick](0.8660254037844386, -1.5)--(-0.8660254037844386, -1.5);
\draw [violet,very thick](-0.8660254037844386, -1.5)--(-1.7320508075688772, 0.0);
\draw [violet,very thick](-1.7320508075688772, 0.0)--(-0.8660254037844386, 1.5);
\draw [darkgreen,very thick](-0.8660254037844386, 2.5)--(0.8660254037844386, 2.5);
\draw [darkgreen,very thick](0.8660254037844386, 2.5)--(2.598076211353316, -0.5);
\draw [darkgreen,very thick](2.598076211353316, -0.5)--(1.7320508075688772, -2.0);
\draw [darkgreen,very thick](1.7320508075688772, -2.0)--(-1.7320508075688772, -2.0);
\draw [darkgreen,very thick](-1.7320508075688772, -2.0)--(-2.598076211353316, -0.5);
\draw [darkgreen,very thick](-2.598076211353316, -0.5)--(-0.8660254037844386, 2.5);
\draw [violet,very thick](-1.7320508075688772, 3.0)--(1.7320508075688772, 3.0);
\draw [violet,very thick](1.7320508075688772, 3.0)--(3.4641016151377544, 0.0);
\draw [violet,very thick](3.4641016151377544, 0.0)--(1.7320508075688772, -3.0);
\draw [violet,very thick](1.7320508075688772, -3.0)--(-1.7320508075688772, -3.0);
\draw [violet,very thick](-1.7320508075688772, -3.0)--(-3.4641016151377544, 0.0);
\draw [violet,very thick](-3.4641016151377544, 0.0)--(-1.7320508075688772, 3.0);
\draw [darkgreen,very thick](-1.7320508075688772, 4.0)--(1.7320508075688772, 4.0);
\draw [darkgreen,very thick](1.7320508075688772, 4.0)--(4.330127018922193, -0.5);
\draw [darkgreen,very thick](4.330127018922193, -0.5)--(2.598076211353316, -3.5);
\draw [darkgreen,very thick](2.598076211353316, -3.5)--(-2.598076211353316, -3.5);
\draw [darkgreen,very thick](-2.598076211353316, -3.5)--(-4.330127018922193, -0.5);
\draw [darkgreen,very thick](-4.330127018922193, -0.5)--(-1.7320508075688772, 4.0);
\draw [violet,very thick](-2.598076211353316, 4.5)--(2.598076211353316, 4.5);
\draw [violet,very thick](2.598076211353316, 4.5)--(5.196152422706632, 0.0);
\draw [violet,very thick](5.196152422706632, 0.0)--(2.598076211353316, -4.5);
\draw [violet,very thick](2.598076211353316, -4.5)--(-2.598076211353316, -4.5);
\draw [violet,very thick](-2.598076211353316, -4.5)--(-5.196152422706632, 0.0);
\draw [violet,very thick](-5.196152422706632, 0.0)--(-2.598076211353316, 4.5);
\draw [red,very thick,-{Latex[length=5mm]}] (0,-6.5)--(0,6.5);
\draw [red,very thick,-{Latex[length=5mm]}] (-7.696152422706632,0)--(7.696152422706632,0);
\node at (0,6.5) [above] {\Large $j$};
\node at (7.696152422706632,0) [right] {\Large $i$};
\end{tikzpicture}
\caption{The $(i,j)$ coordinate system used throughout Section \ref{section hex grid}. $f = (0,0)$ is the red vertex at the center. Its three neighbor vertices, starting from the one directly above, and in clockwise order, are $(0,2),(1,-1)$, and $(-1,1)$.}
\label{Coordinate System}
\end{figure}

In the proof, we often use the distance $\dist(f,v)$ from the initial fire to a given vertex $v$. Let $P_d = \{v\in V\mid\dist(f,v) = d\}$ be the set of vertices $v$ at distance $d$ from $f$. In Figure \ref{Coordinate System}, the sets $P_d$ are marked by green and violet lines for $1 \leq d \leq 6$. Let $(d \bmod 2)$ refer to the remainder of dividing $d$ by $2$. Note that since we embed the hexagonal grid in the plane with $f$ at the origin, we have that the distance $\dist(f,v)=d$ in the hexagonal grid if and only if
\begin{equation}\label{equation distance in hex grid}
\max\left\{\frac{|2j - (d \bmod 2)|}{3},|i|+\frac{|j+(d\bmod 2)|}{3}\right\} = d.
\end{equation}
The proof that~\eqref{equation distance in hex grid} characterizes points at distance $d$ from the origin is straightforward but tedious, so we only give a sketch of the argument here. Let $E_d$ be the points of $V$ satisfying~\eqref{equation distance in hex grid}. It is straightforward to check that the $E_d$ partition $V$, $E_0 = P_0$, and if $u \in E_{d_1}$ is adjacent to $v \in E_{d_2}$, then $|d_1 - d_2| = 1$. Using these facts and induction on $d$, we can show that $E_d = P_d$. One can see that $P_d \subseteq E_d$ by the fact that every vertex of $P_d$ must have a neighbor in $E_{d-1}$, and thus must be in $E_{d-2}$ or $E_d$, but are not in $E_{d-2}$. And to see that $E_d \subseteq P_d$, it suffices to show that every vertex of $E_d$ has a neighbor in $E_{d-1}$. We can do this by checking cases that depend on the parity of $d$, the sign of $i$, and the sign of $j+(d\bmod 2)$. For instance, consider a vertex $(i,j) \in E_d$ with $j<0$ and $d$ even. This implies that $j\bmod 3 = 0$, so $(i,j+2)$ is a neighbor of $(i,j)$. Finally, we see that $(i,j+2) \in E_{d-1}$ since
\begin{align*}
\max&\left\{\frac{|2(j+2) - ((d-1) \bmod 2)|}{3},|i|+\frac{|(j+2)+((d-1)\bmod 2)|}{3}\right\}\\
&= \max\left\{\frac{|2j + 3|}{3},|i|+\frac{|j+3|}{3}\right\} = \max\left\{\frac{|2j|}{3} - 1,|i|+\frac{|j|}{3} - 1 \right\} = d-1 .
\end{align*}
The other cases can be checked similarly.

Our strategy can be broken down into the following steps (see Figure \ref{Strat Outline} for a visual outline):
\begin{enumerate}
\item Before turn $2\tau$, build two protective rays, that if extended indefinitely would protect $\frac{2}{3}$ of the grid.
\item Advance the ray building by one extra step with the extra firefighter at turn $2\tau$. 
\item Bend the protective rays to be parallel to each other. Grow a strip containing the fire with these parallel rays for a sufficiently long time.
\item Bend a ray into a spiral around a vertex $c$. The spiral will collide with the other ray.
\end{enumerate}

\begin{figure}[H]
\centering
    \begin{adjustbox}{max width=0.8\textwidth}
    \begin{tikzpicture}
        \tikzstyle{xdashed}=[dash pattern=on 24pt off 24pt]
	    \draw [black,fill=red] plot[rotate=45,mark=square*,mark size=14pt] (0,0);
	    \node [font=\Huge, scale=3.0] at (-2, 0) {$f$};
	    \draw [line width=2pt] (1,0)--(-1,3.4641);
	    \node [font=\LARGE, scale=3.0] at (1.5, 2.5) {\ref{protecc2/3}};
	    \draw [line width=2pt] (1,0)--(-1,-3.4641);
	    \node [font=\LARGE, scale=3.0] at (1.5, -2.5) {\ref{protecc2/3}};
	    \draw [line width=2pt] (-1,-3.4641)--(-61,-3.4641);
	    \node [font=\Huge, scale=3.0] at (-31, -5) {\ref{protecc}};
	    \draw [line width=2pt] (-1,3.4641)--(-61,3.4641);
	    \node [font=\Huge, scale=3.0] at (-31, 5) {\ref{protecc}};
	    \draw [black,fill=red] plot[rotate=45,mark=square*,mark size=14pt] (-41.7193,41.7193);
	    \node [font=\Huge, scale=3.0] at (-57, 0) {$c$};
	    \draw [color=red,xdashed,line width=2pt] (-59,0)--(-61,-3.4641);
	    \draw [line width=2pt] (-61,-3.4641)--(-67,0);
		\node [font=\Huge, scale=3.0] at (-65.5, -2.598) {$L_1$};
	    \draw [color=red,xdashed,line width=2pt] (-59,0)--(-67,0);
	    \draw [line width=2pt] (-67,0)--(-67,13.8564);
	    \node [font=\Huge, scale=3.0] at (-68.75,6.9282) {$L_2$};
	    \draw [color=red,xdashed,line width=2pt] (-59,0)--(-67,13.8564);
	    \draw [line width=2pt] (-67,13.8564)--(-43,27.7128);
	    \node [font=\Huge, scale=3.0] at (-56.5, 22) {$L_3$};
	    \draw [color=red,xdashed,line width=2pt] (-59,0)--(-43,27.7128);
	    \draw [line width=2pt] (-43,27.7128)--(-1,3.4641);
	    \node [font=\Huge, scale=3.0] at (-21.5, 17) {$L_4$};
    \end{tikzpicture}
    \end{adjustbox}
\caption{Outline of the strategy. All angles are multiples of $30\degree$. The portions of the protective rays built in section \ref{protecc2/3}, that if extended indefinitely would protect $\frac{2}{3}$ of the grid, are labelled \ref{protecc2/3}. Once the rays have been bent to be parallel to each other, they are labelled after section $\ref{protecc}$, where the parallel rays are built from the right to the left and the fire is constrained inside a strip. The bottom ray is eventually bent into a spiral around the vertex $c$, and $L_1,L_2,L_3,L_4$ are the segments of the spiral built in section \ref{spiral}.}
\label{Strat Outline}
\end{figure}
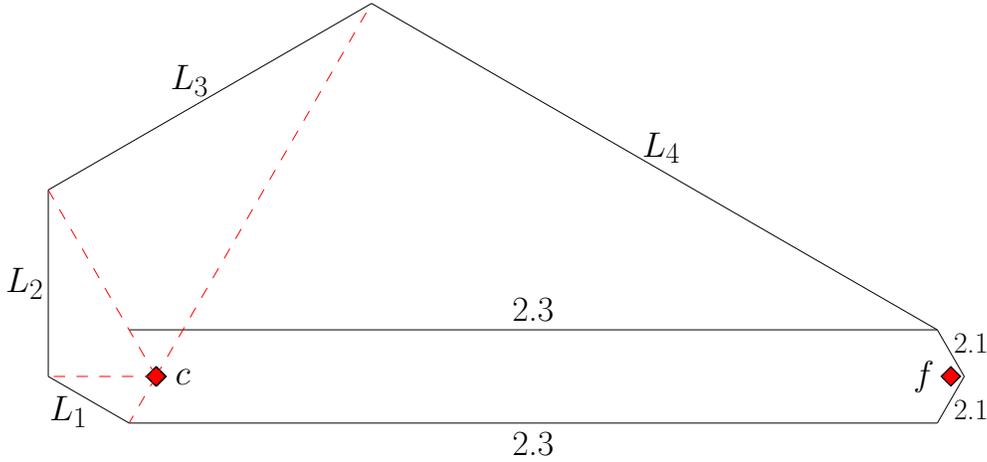 

The improvement over the strategy described in \cite{Grids} is given in the second and third parts. The first and last parts of our strategy are essentially the same as that in \cite{Grids}, but we describe them here for completeness. 

\subsection{Protecting Two-Thirds of the Grid}\label{protecc2/3}
For $0 \leq k \leq \frac{\tau-3}{2}$, on turn $4k+2$, protect $v_{2k+1}:=(1-k, -1-3k)$; on turn $4k+4$, protect $v_{2k+2}:=(1-k, 3+3k)$.

\begin{observation}\label{observation distance protecting two thirds}
Let $v_1,v_2,\dots,v_{\tau-1}$ be the vertices protected in the manner described above. For all $r$ with $1\leq r\leq t-1$, we have $\dist(f,v_r) = r$.
\end{observation}

\begin{proof}
When $r=2k+1$, we protect $v_r$, which is at $(1-k,-1-3k)$. Note that
\[
\max\left\{\frac{|2(-1-3k) - 1|}{3},|1-k|+\frac{|(-1-3k)+1|}{3}\right\} =\max\left\{2k+1,|1-k|+k\right\}=r.
\]
Furthermore, when $r=2k+2$, we protect $(1-k,3+3k)$, and
\[
\max\left\{\frac{|2(3+3k) - 0|}{3},|1-k|+\frac{|(3+3k)+0|}{3}\right\} =\max\left\{2+2k,|1-k|+1+k\right\}=r,
\]
so by Equation \eqref{equation distance in hex grid}, $\dist(f,v_r) = r$.
\end{proof}

Any burning vertex on turn $2j$ is at distance strictly less than $j$ from $f$. And so, by Observation \ref{observation distance protecting two thirds}, we are permitted to protect the vertices described above. Figure \ref{Protect Ray} illustrates this part of strategy.


\subsection{Accelerating the Ray Building}\label{acc}
At turn $2t$ when we receive the extra firefighter, we continue with the ray building strategy of the previous part, but accelerate it with the extra firefighter: We protect the vertices at $(1-\frac{\tau - 1}{2},-1-\frac{3(\tau - 1)}{2})$ and $(1-\frac{\tau - 1}{2},3+\frac{3(\tau - 1)}{2})$ in turn $2t$, instead of just $(1-\frac{\tau - 1}{2},-1-\frac{3(\tau - 1)}{2})$. Figure \ref{Accel Ray} shows this step.
\subsection{Restricting the Fire to a Strip}\label{protecc}

For $0 \leq k \leq \frac{15\tau+11}{2}$, on turn $2\tau+4k+2$ we protect $v_{\tau+2k+1}:=(-\frac{\tau+1}{2}-2k,-\frac{3(\tau+1)}{2})$, and on turn $2\tau+4k+4$ we protect $v_{\tau+2k+2}:=(-\frac{\tau+1}{2}-2k,2+\frac{3(\tau+1)}{2})$. 

\begin{observation}\label{observation distances when restricting to a strip}
Let $v_{\tau+1},v_{\tau+2},\dots,v_{16\tau+13}$ be the vertices protected in the manner described above. For all $r$ with $\tau+1\leq r\leq 16\tau+13$, we have $\dist(f,v_r) = r$.
\end{observation}

\begin{proof}
When $r=\tau+2k+1$, the vertex $v_r$ is at $(-\frac{\tau+1}{2}-2k,-\frac{3(\tau+1)}{2})$, so we have that
\begin{align*}
\max&\left\{\frac{\left|2\left(-3(\tau+1)/2\right) - 0\right|}{3},\left|-\frac{\tau+1}{2}-2k\right|+\frac{\left|\left(-3(\tau+1)/2\right)+0\right|}{3}\right\}\\
&=\max\left\{\tau+1,\tau+2k+1\right\}= r,
\end{align*}
and when $r=\tau+2k+2$, $v_r$ is at $(-\frac{\tau+1}{2}-2k,2+\frac{3(\tau+1)}{2})$, and
\begin{align*}
    \max&\left\{\frac{|2\left(2+3(\tau+1)/2\right) - 1|}{3},\left|-\frac{\tau+1}{2}-2k\right|+\frac{|\left(2+3(\tau+1)/2\right)+1|}{3}\right\}\\
    &=\max\left\{\tau+2,\tau+2k+2\right\}=r.
\end{align*}
Thus, by Equation \eqref{equation distance in hex grid}, $\dist(f,v_r)=r$.
\end{proof}

As before, these moves are permitted since on turn $2j$, we protect a vertex $v$ with $\dist(f,v) = j$. Figure \ref{Strip Begin} shows this effective bending of the rays. We claim that so far, we have constrained the fire to a ``strip''.

\begin{lemma}\label{striplemma}
After the fire spreads at turn $32\tau+27$, a vertex $v=(i,j)$ is on fire if and only if all of the following inequalities hold:
\begin{itemize}
\item $\dist(f,v) \leq 16\tau+13$
\item $-4+3i < j < 6-3i$
\item $\frac{-3(\tau+1)}2 < j < 2+\frac{3(\tau+1)}2$
\end{itemize}
\label{t Lemma}
\end{lemma}

\begin{proof}
After turn $32\tau+27$, the fire has spread exactly $16\tau+13$ times, so the fire is completely contained inside  $B(f,16\tau+13)$, where $B(f,r):=\{v\in V\mid d(f,v)\leq r\}$ is the closed ball of radius $r$ centered around the vertex $f$, which corresponds to the restriction imposed on the first bullet point above. 

In sections \ref{protecc2/3} and \ref{acc}, we protected every vertex of the form $(1-k,-1-3k)$ and $(1-k,3+3k)$, for all $k$ with $0\leq k\leq \frac{(\tau-1)}2$, which corresponds exactly to vertices of the form $(i,j)$ where $j=-4+3i$ and $j=6-3i$ as $i$ ranges from $1-\frac{\tau-1}2$ to $1$. Note that after turn $2\tau$, the fire has spread $\tau-1$ times, so it was completely contained inside $B(f,\tau-1)$, and the vertices we protected in sections \ref{protecc2/3} and \ref{acc} separate $B(f,\tau-1)$ into two regions, with $f$ in the left region, so at turn $2\tau$, the points on fire all satisfy the second bullet point above.

In Section \ref{protecc}, we protected the vertices $(-\frac{\tau+1}{2}-2k,-\frac{3(\tau+1)}{2})$ and $(-\frac{\tau+1}{2}-2k,2+\frac{3(\tau+1)}{2})$ for all $k$ with $0\leq k\leq \frac{15\tau+11}{2}$, which correspond to vertices $(i,j)$ that satisfy the first and second bulletpoints above, and that have $j=\frac{-3(\tau+1)}2$ or $j=2+\frac{3(\tau+1)}{2}$. The set of vertices protected after turn $32\tau+26$ again separate $B(f,16\tau+13)$ into two regions, and the region containing $f$ is characterized by the second and third bulletpoints above, so these points and no other points are on fire after the fire spreads on turn $32\tau+27$.
\end{proof}

Figure \ref{figure building the strip} provides an example of the early part of our strategy when $\tau=5$.

\begin{figure}[H]
\begin{subfigure}{0.20238095238095238\textwidth}
\centering
\usetikzlibrary{plotmarks}
\begin{tikzpicture}[scale=0.4]
\definecolor{lightgray}{RGB}{195,195,195}
\definecolor{smoothred}{RGB}{255,150,150};
\definecolor{smoothgreen}{RGB}{0,105,105};
\draw [very thin,gray] (-6.062178, -6)--(-6.062178, 7);
\draw [very thin,gray] (-5.196152, -6)--(-5.196152, 7);
\draw [very thin,gray] (-4.330127, -6)--(-4.330127, 7);
\draw [very thin,gray] (-3.464102, -6)--(-3.464102, 7);
\draw [very thin,gray] (-2.598076, -6)--(-2.598076, 7);
\draw [very thin,gray] (-1.732051, -6)--(-1.732051, 7);
\draw [very thin,gray] (-0.8660254, -6)--(-0.8660254, 7);
\draw [very thin,gray] (0, -6)--(0, 7);
\draw [very thin,gray] (0.8660254, -6)--(0.8660254, 7);
\draw [very thin,gray] (-6.928203, -5.5)--(1.732051, -5.5);
\draw [very thin,gray] (-6.928203, -5)--(1.732051, -5);
\draw [very thin,gray] (-6.928203, -4.5)--(1.732051, -4.5);
\draw [very thin,gray] (-6.928203, -4)--(1.732051, -4);
\draw [very thin,gray] (-6.928203, -3.5)--(1.732051, -3.5);
\draw [very thin,gray] (-6.928203, -3)--(1.732051, -3);
\draw [very thin,gray] (-6.928203, -2.5)--(1.732051, -2.5);
\draw [very thin,gray] (-6.928203, -2)--(1.732051, -2);
\draw [very thin,gray] (-6.928203, -1.5)--(1.732051, -1.5);
\draw [very thin,gray] (-6.928203, -1)--(1.732051, -1);
\draw [very thin,gray] (-6.928203, -0.5)--(1.732051, -0.5);
\draw [very thin,gray] (-6.928203, 0)--(1.732051, 0);
\draw [very thin,gray] (-6.928203, 0.5)--(1.732051, 0.5);
\draw [very thin,gray] (-6.928203, 1)--(1.732051, 1);
\draw [very thin,gray] (-6.928203, 1.5)--(1.732051, 1.5);
\draw [very thin,gray] (-6.928203, 2)--(1.732051, 2);
\draw [very thin,gray] (-6.928203, 2.5)--(1.732051, 2.5);
\draw [very thin,gray] (-6.928203, 3)--(1.732051, 3);
\draw [very thin,gray] (-6.928203, 3.5)--(1.732051, 3.5);
\draw [very thin,gray] (-6.928203, 4)--(1.732051, 4);
\draw [very thin,gray] (-6.928203, 4.5)--(1.732051, 4.5);
\draw [very thin,gray] (-6.928203, 5)--(1.732051, 5);
\draw [very thin,gray] (-6.928203, 5.5)--(1.732051, 5.5);
\draw [very thin,gray] (-6.928203, 6)--(1.732051, 6);
\draw [very thin,gray] (-6.928203, 6.5)--(1.732051, 6.5);
\draw [very thick] (-0.8660254, -0.5)--(0, 0);
\draw [very thick] (0.8660254, -0.5)--(0, 0);
\draw [very thick] (0, 1)--(0, 0);
\draw [very thick] (-1.732051, 0)--(-0.8660254, -0.5);
\draw [very thick] (-0.8660254, -1.5)--(-0.8660254, -0.5);
\draw [very thick] (1.732051, 0)--(0.8660254, -0.5);
\draw [very thick] (0.8660254, -1.5)--(0.8660254, -0.5);
\draw [very thick] (-0.8660254, 1.5)--(0, 1);
\draw [very thick] (0.8660254, 1.5)--(0, 1);
\draw [very thick] (-2.598076, -0.5)--(-1.732051, 0);
\draw [very thick] (-1.732051, 1)--(-0.8660254, 1.5);
\draw [very thick] (-1.732051, 1)--(-1.732051, 0);
\draw [very thick] (-1.732051, -2)--(-0.8660254, -1.5);
\draw [very thick] (0, -2)--(-0.8660254, -1.5);
\draw [very thick] (0, -2)--(0.8660254, -1.5);
\draw [very thick] (1.732051, 1)--(0.8660254, 1.5);
\draw [very thick] (1.732051, -2)--(0.8660254, -1.5);
\draw [very thick] (-0.8660254, 2.5)--(-0.8660254, 1.5);
\draw [very thick] (0.8660254, 2.5)--(0.8660254, 1.5);
\draw [very thick] (-3.464102, 0)--(-2.598076, -0.5);
\draw [very thick] (-2.598076, -1.5)--(-1.732051, -2);
\draw [very thick] (-2.598076, -1.5)--(-2.598076, -0.5);
\draw [very thick] (-2.598076, 1.5)--(-1.732051, 1);
\draw [very thick] (-1.732051, -3)--(-1.732051, -2);
\draw [very thick] (0, -3)--(0, -2);
\draw [very thick] (-1.732051, 3)--(-0.8660254, 2.5);
\draw [very thick] (0, 3)--(-0.8660254, 2.5);
\draw [very thick] (0, 3)--(0.8660254, 2.5);
\draw [very thick] (1.732051, 3)--(0.8660254, 2.5);
\draw [very thick] (-4.330127, -0.5)--(-3.464102, 0);
\draw [very thick] (-3.464102, 1)--(-2.598076, 1.5);
\draw [very thick] (-3.464102, 1)--(-3.464102, 0);
\draw [very thick] (-3.464102, -2)--(-2.598076, -1.5);
\draw [very thick] (-2.598076, 2.5)--(-1.732051, 3);
\draw [very thick] (-2.598076, 2.5)--(-2.598076, 1.5);
\draw [very thick] (-2.598076, -3.5)--(-1.732051, -3);
\draw [very thick] (-0.8660254, -3.5)--(-1.732051, -3);
\draw [very thick] (-0.8660254, -3.5)--(0, -3);
\draw [very thick] (0.8660254, -3.5)--(0, -3);
\draw [very thick] (0.8660254, -3.5)--(1.732051, -3);
\draw [very thick] (-1.732051, 4)--(-1.732051, 3);
\draw [very thick] (0, 4)--(0, 3);
\draw [very thick] (-5.196152, 0)--(-4.330127, -0.5);
\draw [very thick] (-4.330127, -1.5)--(-3.464102, -2);
\draw [very thick] (-4.330127, -1.5)--(-4.330127, -0.5);
\draw [very thick] (-4.330127, 1.5)--(-3.464102, 1);
\draw [very thick] (-3.464102, -3)--(-2.598076, -3.5);
\draw [very thick] (-3.464102, -3)--(-3.464102, -2);
\draw [very thick] (-3.464102, 3)--(-2.598076, 2.5);
\draw [very thick] (-2.598076, -4.5)--(-2.598076, -3.5);
\draw [very thick] (-0.8660254, -4.5)--(-0.8660254, -3.5);
\draw [very thick] (0.8660254, -4.5)--(0.8660254, -3.5);
\draw [very thick] (-2.598076, 4.5)--(-1.732051, 4);
\draw [very thick] (-0.8660254, 4.5)--(-1.732051, 4);
\draw [very thick] (-0.8660254, 4.5)--(0, 4);
\draw [very thick] (0.8660254, 4.5)--(0, 4);
\draw [very thick] (0.8660254, 4.5)--(1.732051, 4);
\draw [very thick] (-6.062178, -0.5)--(-5.196152, 0);
\draw [very thick] (-5.196152, 1)--(-4.330127, 1.5);
\draw [very thick] (-5.196152, 1)--(-5.196152, 0);
\draw [very thick] (-5.196152, -2)--(-4.330127, -1.5);
\draw [very thick] (-4.330127, 2.5)--(-3.464102, 3);
\draw [very thick] (-4.330127, 2.5)--(-4.330127, 1.5);
\draw [very thick] (-4.330127, -3.5)--(-3.464102, -3);
\draw [very thick] (-3.464102, 4)--(-2.598076, 4.5);
\draw [very thick] (-3.464102, 4)--(-3.464102, 3);
\draw [very thick] (-3.464102, -5)--(-2.598076, -4.5);
\draw [very thick] (-1.732051, -5)--(-2.598076, -4.5);
\draw [very thick] (-1.732051, -5)--(-0.8660254, -4.5);
\draw [very thick] (0, -5)--(-0.8660254, -4.5);
\draw [very thick] (0, -5)--(0.8660254, -4.5);
\draw [very thick] (1.732051, -5)--(0.8660254, -4.5);
\draw [very thick] (-2.598076, 5.5)--(-2.598076, 4.5);
\draw [very thick] (-0.8660254, 5.5)--(-0.8660254, 4.5);
\draw [very thick] (0.8660254, 5.5)--(0.8660254, 4.5);
\draw [very thick] (-6.928203, 0)--(-6.062178, -0.5);
\draw [very thick] (-6.062178, -1.5)--(-5.196152, -2);
\draw [very thick] (-6.062178, -1.5)--(-6.062178, -0.5);
\draw [very thick] (-6.062178, 1.5)--(-5.196152, 1);
\draw [very thick] (-5.196152, -3)--(-4.330127, -3.5);
\draw [very thick] (-5.196152, -3)--(-5.196152, -2);
\draw [very thick] (-5.196152, 3)--(-4.330127, 2.5);
\draw [very thick] (-4.330127, -4.5)--(-3.464102, -5);
\draw [very thick] (-4.330127, -4.5)--(-4.330127, -3.5);
\draw [very thick] (-4.330127, 4.5)--(-3.464102, 4);
\draw [very thick] (-3.464102, -6)--(-3.464102, -5);
\draw [very thick] (-1.732051, -6)--(-1.732051, -5);
\draw [very thick] (0, -6)--(0, -5);
\draw [very thick] (-3.464102, 6)--(-2.598076, 5.5);
\draw [very thick] (-1.732051, 6)--(-2.598076, 5.5);
\draw [very thick] (-1.732051, 6)--(-0.8660254, 5.5);
\draw [very thick] (0, 6)--(-0.8660254, 5.5);
\draw [very thick] (0, 6)--(0.8660254, 5.5);
\draw [very thick] (1.732051, 6)--(0.8660254, 5.5);
\draw [very thick] (-6.928203, 1)--(-6.062178, 1.5);
\draw [very thick] (-6.928203, -2)--(-6.062178, -1.5);
\draw [very thick] (-6.062178, 2.5)--(-5.196152, 3);
\draw [very thick] (-6.062178, 2.5)--(-6.062178, 1.5);
\draw [very thick] (-6.062178, -3.5)--(-5.196152, -3);
\draw [very thick] (-5.196152, 4)--(-4.330127, 4.5);
\draw [very thick] (-5.196152, 4)--(-5.196152, 3);
\draw [very thick] (-5.196152, -5)--(-4.330127, -4.5);
\draw [very thick] (-4.330127, 5.5)--(-3.464102, 6);
\draw [very thick] (-4.330127, 5.5)--(-4.330127, 4.5);
\draw [very thick] (-3.464102, 7)--(-3.464102, 6);
\draw [very thick] (-1.732051, 7)--(-1.732051, 6);
\draw [very thick] (0, 7)--(0, 6);
\draw [very thick] (-6.928203, -3)--(-6.062178, -3.5);
\draw [very thick] (-6.928203, 3)--(-6.062178, 2.5);
\draw [very thick] (-6.062178, -4.5)--(-5.196152, -5);
\draw [very thick] (-6.062178, -4.5)--(-6.062178, -3.5);
\draw [very thick] (-6.062178, 4.5)--(-5.196152, 4);
\draw [very thick] (-5.196152, -6)--(-5.196152, -5);
\draw [very thick] (-5.196152, 6)--(-4.330127, 5.5);
\draw [very thick] (-6.928203, 4)--(-6.062178, 4.5);
\draw [very thick] (-6.928203, -5)--(-6.062178, -4.5);
\draw [very thick] (-6.062178, 5.5)--(-5.196152, 6);
\draw [very thick] (-6.062178, 5.5)--(-6.062178, 4.5);
\draw [very thick] (-5.196152, 7)--(-5.196152, 6);
\draw [very thick] (-6.928203, 6)--(-6.062178, 5.5);
\draw [black,fill=red,rotate around={45:(0, 0)}] plot[mark=square*,mark size=5.4pt] (0, 0);
\draw [black,fill=smoothred,rotate around={45:(-0.8660254, -0.5)}] plot[mark=square*,mark size=5.4pt] (-0.8660254, -0.5);
\draw [black,fill=smoothgreen] plot[mark=square*,mark size=5.4pt] (0.8660254, -0.5);
\draw [black,fill=smoothred,rotate around={45:(0, 1)}] plot[mark=square*,mark size=5.4pt] (0, 1);
\draw [black,fill=smoothred,rotate around={45:(-1.732051, 0)}] plot[mark=square*,mark size=5.4pt] (-1.732051, 0);
\draw [black,fill=smoothred,rotate around={45:(-0.8660254, -1.5)}] plot[mark=square*,mark size=5.4pt] (-0.8660254, -1.5);
\draw [fill=lightgray] (0.8660254, -1.5) circle (5.4pt);
\draw [black,fill=smoothred,rotate around={45:(-0.8660254, 1.5)}] plot[mark=square*,mark size=5.4pt] (-0.8660254, 1.5);
\draw [black,fill=smoothgreen] plot[mark=square*,mark size=5.4pt] (0.8660254, 1.5);
\draw [black,fill=smoothred,rotate around={45:(-2.598076, -0.5)}] plot[mark=square*,mark size=5.4pt] (-2.598076, -0.5);
\draw [black,fill=smoothred,rotate around={45:(-1.732051, 1)}] plot[mark=square*,mark size=5.4pt] (-1.732051, 1);
\draw [black,fill=smoothred,rotate around={45:(-1.732051, -2)}] plot[mark=square*,mark size=5.4pt] (-1.732051, -2);
\draw [black,fill=smoothgreen] plot[mark=square*,mark size=5.4pt] (0, -2);
\draw [black,fill=smoothred,rotate around={45:(-0.8660254, 2.5)}] plot[mark=square*,mark size=5.4pt] (-0.8660254, 2.5);
\draw [fill=lightgray] (0.8660254, 2.5) circle (5.4pt);
\draw [black,fill=smoothred,rotate around={45:(-3.464102, 0)}] plot[mark=square*,mark size=5.4pt] (-3.464102, 0);
\draw [black,fill=smoothred,rotate around={45:(-2.598076, -1.5)}] plot[mark=square*,mark size=5.4pt] (-2.598076, -1.5);
\draw [black,fill=smoothred,rotate around={45:(-2.598076, 1.5)}] plot[mark=square*,mark size=5.4pt] (-2.598076, 1.5);
\draw [black,fill=smoothred,rotate around={45:(-1.732051, -3)}] plot[mark=square*,mark size=5.4pt] (-1.732051, -3);
\draw [fill=lightgray] (0, -3) circle (5.4pt);
\draw [black,fill=smoothred,rotate around={45:(-1.732051, 3)}] plot[mark=square*,mark size=5.4pt] (-1.732051, 3);
\draw [black,fill=smoothgreen] plot[mark=square*,mark size=5.4pt] (0, 3);
\draw [fill=lightgray] (-4.330127, -0.5) circle (5.4pt);
\draw [fill=lightgray] (-3.464102, 1) circle (5.4pt);
\draw [fill=lightgray] (-3.464102, -2) circle (5.4pt);
\draw [fill=lightgray] (-2.598076, 2.5) circle (5.4pt);
\draw [fill=lightgray] (-2.598076, -3.5) circle (5.4pt);
\draw [fill=lightgray] (-0.8660254, -3.5) circle (5.4pt);
\draw [fill=lightgray] (0.8660254, -3.5) circle (5.4pt);
\draw [fill=lightgray] (-1.732051, 4) circle (5.4pt);
\draw [fill=lightgray] (0, 4) circle (5.4pt);
\draw [fill=lightgray] (-5.196152, 0) circle (5.4pt);
\draw [fill=lightgray] (-4.330127, -1.5) circle (5.4pt);
\draw [fill=lightgray] (-4.330127, 1.5) circle (5.4pt);
\draw [fill=lightgray] (-3.464102, -3) circle (5.4pt);
\draw [fill=lightgray] (-3.464102, 3) circle (5.4pt);
\draw [fill=lightgray] (-2.598076, -4.5) circle (5.4pt);
\draw [fill=lightgray] (-0.8660254, -4.5) circle (5.4pt);
\draw [fill=lightgray] (0.8660254, -4.5) circle (5.4pt);
\draw [fill=lightgray] (-2.598076, 4.5) circle (5.4pt);
\draw [fill=lightgray] (-0.8660254, 4.5) circle (5.4pt);
\draw [fill=lightgray] (0.8660254, 4.5) circle (5.4pt);
\draw [fill=lightgray] (-6.062178, -0.5) circle (5.4pt);
\draw [fill=lightgray] (-5.196152, 1) circle (5.4pt);
\draw [fill=lightgray] (-5.196152, -2) circle (5.4pt);
\draw [fill=lightgray] (-4.330127, 2.5) circle (5.4pt);
\draw [fill=lightgray] (-4.330127, -3.5) circle (5.4pt);
\draw [fill=lightgray] (-3.464102, 4) circle (5.4pt);
\draw [fill=lightgray] (-3.464102, -5) circle (5.4pt);
\draw [fill=lightgray] (-1.732051, -5) circle (5.4pt);
\draw [fill=lightgray] (0, -5) circle (5.4pt);
\draw [fill=lightgray] (-2.598076, 5.5) circle (5.4pt);
\draw [fill=lightgray] (-0.8660254, 5.5) circle (5.4pt);
\draw [fill=lightgray] (0.8660254, 5.5) circle (5.4pt);
\draw [fill=lightgray] (-6.062178, -1.5) circle (5.4pt);
\draw [fill=lightgray] (-6.062178, 1.5) circle (5.4pt);
\draw [fill=lightgray] (-5.196152, -3) circle (5.4pt);
\draw [fill=lightgray] (-5.196152, 3) circle (5.4pt);
\draw [fill=lightgray] (-4.330127, -4.5) circle (5.4pt);
\draw [fill=lightgray] (-4.330127, 4.5) circle (5.4pt);
\draw [fill=lightgray] (-3.464102, 6) circle (5.4pt);
\draw [fill=lightgray] (-1.732051, 6) circle (5.4pt);
\draw [fill=lightgray] (0, 6) circle (5.4pt);
\draw [fill=lightgray] (-6.062178, 2.5) circle (5.4pt);
\draw [fill=lightgray] (-6.062178, -3.5) circle (5.4pt);
\draw [fill=lightgray] (-5.196152, 4) circle (5.4pt);
\draw [fill=lightgray] (-5.196152, -5) circle (5.4pt);
\draw [fill=lightgray] (-4.330127, 5.5) circle (5.4pt);
\draw [fill=lightgray] (-6.062178, -4.5) circle (5.4pt);
\draw [fill=lightgray] (-6.062178, 4.5) circle (5.4pt);
\draw [fill=lightgray] (-5.196152, 6) circle (5.4pt);
\draw [fill=lightgray] (-6.062178, 5.5) circle (5.4pt);
\draw [red,fill opacity=0,line width=1.8pt] (0, 0) circle (14.4pt);
\end{tikzpicture}
\caption{Turn $2\tau - 1$.}
\label{Protect Ray}
\end{subfigure}
\hspace{0.05\textwidth}
\hspace*{-1em}
\begin{subfigure}{0.20238095238095238\textwidth}
\centering
\begin{tikzpicture}[scale=0.4]
\definecolor{lightgray}{RGB}{195,195,195}
\definecolor{smoothred}{RGB}{255,150,150};
\definecolor{smoothgreen}{RGB}{0,105,105};
\draw [gray,very thin] (-6.062178, -6)--(-6.062178, 7);
\draw [gray,very thin] (-5.196152, -6)--(-5.196152, 7);
\draw [gray,very thin] (-4.330127, -6)--(-4.330127, 7);
\draw [gray,very thin] (-3.464102, -6)--(-3.464102, 7);
\draw [gray,very thin] (-2.598076, -6)--(-2.598076, 7);
\draw [gray,very thin] (-1.732051, -6)--(-1.732051, 7);
\draw [gray,very thin] (-0.8660254, -6)--(-0.8660254, 7);
\draw [gray,very thin] (0, -6)--(0, 7);
\draw [gray,very thin] (0.8660254, -6)--(0.8660254, 7);
\draw [gray,very thin] (-6.928203, -5.5)--(1.732051, -5.5);
\draw [gray,very thin] (-6.928203, -5)--(1.732051, -5);
\draw [gray,very thin] (-6.928203, -4.5)--(1.732051, -4.5);
\draw [gray,very thin] (-6.928203, -4)--(1.732051, -4);
\draw [gray,very thin] (-6.928203, -3.5)--(1.732051, -3.5);
\draw [gray,very thin] (-6.928203, -3)--(1.732051, -3);
\draw [gray,very thin] (-6.928203, -2.5)--(1.732051, -2.5);
\draw [gray,very thin] (-6.928203, -2)--(1.732051, -2);
\draw [gray,very thin] (-6.928203, -1.5)--(1.732051, -1.5);
\draw [gray,very thin] (-6.928203, -1)--(1.732051, -1);
\draw [gray,very thin] (-6.928203, -0.5)--(1.732051, -0.5);
\draw [gray,very thin] (-6.928203, 0)--(1.732051, 0);
\draw [gray,very thin] (-6.928203, 0.5)--(1.732051, 0.5);
\draw [gray,very thin] (-6.928203, 1)--(1.732051, 1);
\draw [gray,very thin] (-6.928203, 1.5)--(1.732051, 1.5);
\draw [gray,very thin] (-6.928203, 2)--(1.732051, 2);
\draw [gray,very thin] (-6.928203, 2.5)--(1.732051, 2.5);
\draw [gray,very thin] (-6.928203, 3)--(1.732051, 3);
\draw [gray,very thin] (-6.928203, 3.5)--(1.732051, 3.5);
\draw [gray,very thin] (-6.928203, 4)--(1.732051, 4);
\draw [gray,very thin] (-6.928203, 4.5)--(1.732051, 4.5);
\draw [gray,very thin] (-6.928203, 5)--(1.732051, 5);
\draw [gray,very thin] (-6.928203, 5.5)--(1.732051, 5.5);
\draw [gray,very thin] (-6.928203, 6)--(1.732051, 6);
\draw [gray,very thin] (-6.928203, 6.5)--(1.732051, 6.5);
\draw [very thick] (-0.8660254, -0.5)--(0, 0);
\draw [very thick] (0.8660254, -0.5)--(0, 0);
\draw [very thick] (0, 1)--(0, 0);
\draw [very thick] (-1.732051, 0)--(-0.8660254, -0.5);
\draw [very thick] (-0.8660254, -1.5)--(-0.8660254, -0.5);
\draw [very thick] (1.732051, 0)--(0.8660254, -0.5);
\draw [very thick] (0.8660254, -1.5)--(0.8660254, -0.5);
\draw [very thick] (-0.8660254, 1.5)--(0, 1);
\draw [very thick] (0.8660254, 1.5)--(0, 1);
\draw [very thick] (-2.598076, -0.5)--(-1.732051, 0);
\draw [very thick] (-1.732051, 1)--(-0.8660254, 1.5);
\draw [very thick] (-1.732051, 1)--(-1.732051, 0);
\draw [very thick] (-1.732051, -2)--(-0.8660254, -1.5);
\draw [very thick] (0, -2)--(-0.8660254, -1.5);
\draw [very thick] (0, -2)--(0.8660254, -1.5);
\draw [very thick] (1.732051, 1)--(0.8660254, 1.5);
\draw [very thick] (1.732051, -2)--(0.8660254, -1.5);
\draw [very thick] (-0.8660254, 2.5)--(-0.8660254, 1.5);
\draw [very thick] (0.8660254, 2.5)--(0.8660254, 1.5);
\draw [very thick] (-3.464102, 0)--(-2.598076, -0.5);
\draw [very thick] (-2.598076, -1.5)--(-1.732051, -2);
\draw [very thick] (-2.598076, -1.5)--(-2.598076, -0.5);
\draw [very thick] (-2.598076, 1.5)--(-1.732051, 1);
\draw [very thick] (-1.732051, -3)--(-1.732051, -2);
\draw [very thick] (0, -3)--(0, -2);
\draw [very thick] (-1.732051, 3)--(-0.8660254, 2.5);
\draw [very thick] (0, 3)--(-0.8660254, 2.5);
\draw [very thick] (0, 3)--(0.8660254, 2.5);
\draw [very thick] (1.732051, 3)--(0.8660254, 2.5);
\draw [very thick] (-4.330127, -0.5)--(-3.464102, 0);
\draw [very thick] (-3.464102, 1)--(-2.598076, 1.5);
\draw [very thick] (-3.464102, 1)--(-3.464102, 0);
\draw [very thick] (-3.464102, -2)--(-2.598076, -1.5);
\draw [very thick] (-2.598076, 2.5)--(-1.732051, 3);
\draw [very thick] (-2.598076, 2.5)--(-2.598076, 1.5);
\draw [very thick] (-2.598076, -3.5)--(-1.732051, -3);
\draw [very thick] (-0.8660254, -3.5)--(-1.732051, -3);
\draw [very thick] (-0.8660254, -3.5)--(0, -3);
\draw [very thick] (0.8660254, -3.5)--(0, -3);
\draw [very thick] (0.8660254, -3.5)--(1.732051, -3);
\draw [very thick] (-1.732051, 4)--(-1.732051, 3);
\draw [very thick] (0, 4)--(0, 3);
\draw [very thick] (-5.196152, 0)--(-4.330127, -0.5);
\draw [very thick] (-4.330127, -1.5)--(-3.464102, -2);
\draw [very thick] (-4.330127, -1.5)--(-4.330127, -0.5);
\draw [very thick] (-4.330127, 1.5)--(-3.464102, 1);
\draw [very thick] (-3.464102, -3)--(-2.598076, -3.5);
\draw [very thick] (-3.464102, -3)--(-3.464102, -2);
\draw [very thick] (-3.464102, 3)--(-2.598076, 2.5);
\draw [very thick] (-2.598076, -4.5)--(-2.598076, -3.5);
\draw [very thick] (-0.8660254, -4.5)--(-0.8660254, -3.5);
\draw [very thick] (0.8660254, -4.5)--(0.8660254, -3.5);
\draw [very thick] (-2.598076, 4.5)--(-1.732051, 4);
\draw [very thick] (-0.8660254, 4.5)--(-1.732051, 4);
\draw [very thick] (-0.8660254, 4.5)--(0, 4);
\draw [very thick] (0.8660254, 4.5)--(0, 4);
\draw [very thick] (0.8660254, 4.5)--(1.732051, 4);
\draw [very thick] (-6.062178, -0.5)--(-5.196152, 0);
\draw [very thick] (-5.196152, 1)--(-4.330127, 1.5);
\draw [very thick] (-5.196152, 1)--(-5.196152, 0);
\draw [very thick] (-5.196152, -2)--(-4.330127, -1.5);
\draw [very thick] (-4.330127, 2.5)--(-3.464102, 3);
\draw [very thick] (-4.330127, 2.5)--(-4.330127, 1.5);
\draw [very thick] (-4.330127, -3.5)--(-3.464102, -3);
\draw [very thick] (-3.464102, 4)--(-2.598076, 4.5);
\draw [very thick] (-3.464102, 4)--(-3.464102, 3);
\draw [very thick] (-3.464102, -5)--(-2.598076, -4.5);
\draw [very thick] (-1.732051, -5)--(-2.598076, -4.5);
\draw [very thick] (-1.732051, -5)--(-0.8660254, -4.5);
\draw [very thick] (0, -5)--(-0.8660254, -4.5);
\draw [very thick] (0, -5)--(0.8660254, -4.5);
\draw [very thick] (1.732051, -5)--(0.8660254, -4.5);
\draw [very thick] (-2.598076, 5.5)--(-2.598076, 4.5);
\draw [very thick] (-0.8660254, 5.5)--(-0.8660254, 4.5);
\draw [very thick] (0.8660254, 5.5)--(0.8660254, 4.5);
\draw [very thick] (-6.928203, 0)--(-6.062178, -0.5);
\draw [very thick] (-6.062178, -1.5)--(-5.196152, -2);
\draw [very thick] (-6.062178, -1.5)--(-6.062178, -0.5);
\draw [very thick] (-6.062178, 1.5)--(-5.196152, 1);
\draw [very thick] (-5.196152, -3)--(-4.330127, -3.5);
\draw [very thick] (-5.196152, -3)--(-5.196152, -2);
\draw [very thick] (-5.196152, 3)--(-4.330127, 2.5);
\draw [very thick] (-4.330127, -4.5)--(-3.464102, -5);
\draw [very thick] (-4.330127, -4.5)--(-4.330127, -3.5);
\draw [very thick] (-4.330127, 4.5)--(-3.464102, 4);
\draw [very thick] (-3.464102, -6)--(-3.464102, -5);
\draw [very thick] (-1.732051, -6)--(-1.732051, -5);
\draw [very thick] (0, -6)--(0, -5);
\draw [very thick] (-3.464102, 6)--(-2.598076, 5.5);
\draw [very thick] (-1.732051, 6)--(-2.598076, 5.5);
\draw [very thick] (-1.732051, 6)--(-0.8660254, 5.5);
\draw [very thick] (0, 6)--(-0.8660254, 5.5);
\draw [very thick] (0, 6)--(0.8660254, 5.5);
\draw [very thick] (1.732051, 6)--(0.8660254, 5.5);
\draw [very thick] (-6.928203, 1)--(-6.062178, 1.5);
\draw [very thick] (-6.928203, -2)--(-6.062178, -1.5);
\draw [very thick] (-6.062178, 2.5)--(-5.196152, 3);
\draw [very thick] (-6.062178, 2.5)--(-6.062178, 1.5);
\draw [very thick] (-6.062178, -3.5)--(-5.196152, -3);
\draw [very thick] (-5.196152, 4)--(-4.330127, 4.5);
\draw [very thick] (-5.196152, 4)--(-5.196152, 3);
\draw [very thick] (-5.196152, -5)--(-4.330127, -4.5);
\draw [very thick] (-4.330127, 5.5)--(-3.464102, 6);
\draw [very thick] (-4.330127, 5.5)--(-4.330127, 4.5);
\draw [very thick] (-3.464102, 7)--(-3.464102, 6);
\draw [very thick] (-1.732051, 7)--(-1.732051, 6);
\draw [very thick] (0, 7)--(0, 6);
\draw [very thick] (-6.928203, -3)--(-6.062178, -3.5);
\draw [very thick] (-6.928203, 3)--(-6.062178, 2.5);
\draw [very thick] (-6.062178, -4.5)--(-5.196152, -5);
\draw [very thick] (-6.062178, -4.5)--(-6.062178, -3.5);
\draw [very thick] (-6.062178, 4.5)--(-5.196152, 4);
\draw [very thick] (-5.196152, -6)--(-5.196152, -5);
\draw [very thick] (-5.196152, 6)--(-4.330127, 5.5);
\draw [very thick] (-6.928203, 4)--(-6.062178, 4.5);
\draw [very thick] (-6.928203, -5)--(-6.062178, -4.5);
\draw [very thick] (-6.062178, 5.5)--(-5.196152, 6);
\draw [very thick] (-6.062178, 5.5)--(-6.062178, 4.5);
\draw [very thick] (-5.196152, 7)--(-5.196152, 6);
\draw [very thick] (-6.928203, 6)--(-6.062178, 5.5);
\draw [black,fill=red,rotate around={45:(0, 0)}] plot[mark=square*,mark size=5.4pt] (0, 0);
\draw [black,fill=smoothred,rotate around={45:(-0.8660254, -0.5)}] plot[mark=square*,mark size=5.4pt] (-0.8660254, -0.5);
\draw [black,fill=smoothgreen] plot[mark=square*,mark size=5.4pt] (0.8660254, -0.5);
\draw [black,fill=smoothred,rotate around={45:(0, 1)}] plot[mark=square*,mark size=5.4pt] (0, 1);
\draw [black,fill=smoothred,rotate around={45:(-1.732051, 0)}] plot[mark=square*,mark size=5.4pt] (-1.732051, 0);
\draw [black,fill=smoothred,rotate around={45:(-0.8660254, -1.5)}] plot[mark=square*,mark size=5.4pt] (-0.8660254, -1.5);
\draw [fill=lightgray] (0.8660254, -1.5) circle (5.4pt);
\draw [black,fill=smoothred,rotate around={45:(-0.8660254, 1.5)}] plot[mark=square*,mark size=5.4pt] (-0.8660254, 1.5);
\draw [black,fill=smoothgreen] plot[mark=square*,mark size=5.4pt] (0.8660254, 1.5);
\draw [black,fill=smoothred,rotate around={45:(-2.598076, -0.5)}] plot[mark=square*,mark size=5.4pt] (-2.598076, -0.5);
\draw [black,fill=smoothred,rotate around={45:(-1.732051, 1)}] plot[mark=square*,mark size=5.4pt] (-1.732051, 1);
\draw [black,fill=smoothred,rotate around={45:(-1.732051, -2)}] plot[mark=square*,mark size=5.4pt] (-1.732051, -2);
\draw [black,fill=smoothgreen] plot[mark=square*,mark size=5.4pt] (0, -2);
\draw [black,fill=smoothred,rotate around={45:(-0.8660254, 2.5)}] plot[mark=square*,mark size=5.4pt] (-0.8660254, 2.5);
\draw [fill=lightgray] (0.8660254, 2.5) circle (5.4pt);
\draw [black,fill=smoothred,rotate around={45:(-3.464102, 0)}] plot[mark=square*,mark size=5.4pt] (-3.464102, 0);
\draw [black,fill=smoothred,rotate around={45:(-2.598076, -1.5)}] plot[mark=square*,mark size=5.4pt] (-2.598076, -1.5);
\draw [black,fill=smoothred,rotate around={45:(-2.598076, 1.5)}] plot[mark=square*,mark size=5.4pt] (-2.598076, 1.5);
\draw [black,fill=smoothred,rotate around={45:(-1.732051, -3)}] plot[mark=square*,mark size=5.4pt] (-1.732051, -3);
\draw [fill=lightgray] (0, -3) circle (5.4pt);
\draw [black,fill=smoothred,rotate around={45:(-1.732051, 3)}] plot[mark=square*,mark size=5.4pt] (-1.732051, 3);
\draw [black,fill=smoothgreen] plot[mark=square*,mark size=5.4pt] (0, 3);
\draw [black,fill=smoothred,rotate around={45:(-4.330127, -0.5)}] plot[mark=square*,mark size=5.4pt] (-4.330127, -0.5);
\draw [black,fill=smoothred,rotate around={45:(-3.464102, 1)}] plot[mark=square*,mark size=5.4pt] (-3.464102, 1);
\draw [black,fill=smoothred,rotate around={45:(-3.464102, -2)}] plot[mark=square*,mark size=5.4pt] (-3.464102, -2);
\draw [black,fill=smoothred,rotate around={45:(-2.598076, 2.5)}] plot[mark=square*,mark size=5.4pt] (-2.598076, 2.5);
\draw [black,fill=smoothred,rotate around={45:(-2.598076, -3.5)}] plot[mark=square*,mark size=5.4pt] (-2.598076, -3.5);
\draw [black,fill=smoothgreen] plot[mark=square*,mark size=5.4pt] (-0.8660254, -3.5);
\draw [fill=lightgray] (0.8660254, -3.5) circle (5.4pt);
\draw [black,fill=smoothred,rotate around={45:(-1.732051, 4)}] plot[mark=square*,mark size=5.4pt] (-1.732051, 4);
\draw [fill=lightgray] (0, 4) circle (5.4pt);
\draw [fill=lightgray] (-5.196152, 0) circle (5.4pt);
\draw [fill=lightgray] (-4.330127, -1.5) circle (5.4pt);
\draw [fill=lightgray] (-4.330127, 1.5) circle (5.4pt);
\draw [fill=lightgray] (-3.464102, -3) circle (5.4pt);
\draw [fill=lightgray] (-3.464102, 3) circle (5.4pt);
\draw [fill=lightgray] (-2.598076, -4.5) circle (5.4pt);
\draw [fill=lightgray] (-0.8660254, -4.5) circle (5.4pt);
\draw [fill=lightgray] (0.8660254, -4.5) circle (5.4pt);
\draw [fill=lightgray] (-2.598076, 4.5) circle (5.4pt);
\draw [black,fill=smoothgreen] plot[mark=square*,mark size=5.4pt] (-0.8660254, 4.5);
\draw [fill=lightgray] (0.8660254, 4.5) circle (5.4pt);
\draw [fill=lightgray] (-6.062178, -0.5) circle (5.4pt);
\draw [fill=lightgray] (-5.196152, 1) circle (5.4pt);
\draw [fill=lightgray] (-5.196152, -2) circle (5.4pt);
\draw [fill=lightgray] (-4.330127, 2.5) circle (5.4pt);
\draw [fill=lightgray] (-4.330127, -3.5) circle (5.4pt);
\draw [fill=lightgray] (-3.464102, 4) circle (5.4pt);
\draw [fill=lightgray] (-3.464102, -5) circle (5.4pt);
\draw [fill=lightgray] (-1.732051, -5) circle (5.4pt);
\draw [fill=lightgray] (0, -5) circle (5.4pt);
\draw [fill=lightgray] (-2.598076, 5.5) circle (5.4pt);
\draw [fill=lightgray] (-0.8660254, 5.5) circle (5.4pt);
\draw [fill=lightgray] (0.8660254, 5.5) circle (5.4pt);
\draw [fill=lightgray] (-6.062178, -1.5) circle (5.4pt);
\draw [fill=lightgray] (-6.062178, 1.5) circle (5.4pt);
\draw [fill=lightgray] (-5.196152, -3) circle (5.4pt);
\draw [fill=lightgray] (-5.196152, 3) circle (5.4pt);
\draw [fill=lightgray] (-4.330127, -4.5) circle (5.4pt);
\draw [fill=lightgray] (-4.330127, 4.5) circle (5.4pt);
\draw [fill=lightgray] (-3.464102, 6) circle (5.4pt);
\draw [fill=lightgray] (-1.732051, 6) circle (5.4pt);
\draw [fill=lightgray] (0, 6) circle (5.4pt);
\draw [fill=lightgray] (-6.062178, 2.5) circle (5.4pt);
\draw [fill=lightgray] (-6.062178, -3.5) circle (5.4pt);
\draw [fill=lightgray] (-5.196152, 4) circle (5.4pt);
\draw [fill=lightgray] (-5.196152, -5) circle (5.4pt);
\draw [fill=lightgray] (-4.330127, 5.5) circle (5.4pt);
\draw [fill=lightgray] (-6.062178, -4.5) circle (5.4pt);
\draw [fill=lightgray] (-6.062178, 4.5) circle (5.4pt);
\draw [fill=lightgray] (-5.196152, 6) circle (5.4pt);
\draw [fill=lightgray] (-6.062178, 5.5) circle (5.4pt);
\draw [red,fill opacity=0,line width=1.8pt] (0, 0) circle (14.4pt);
\end{tikzpicture}
\caption{Turn $2\tau + 1$.}
\label{Accel Ray}
\end{subfigure}
\hspace{0.05\textwidth}
\hspace*{-1em}
\begin{subfigure}{0.20238095238095238\textwidth}
\begin{tikzpicture}[scale=0.4]
\definecolor{lightgray}{RGB}{195,195,195}
\definecolor{smoothred}{RGB}{255,150,150};
\definecolor{smoothgreen}{RGB}{0,105,105};
\draw [gray,very thin] (-6.062178, -6)--(-6.062178, 7);
\draw [gray,very thin] (-5.196152, -6)--(-5.196152, 7);
\draw [gray,very thin] (-4.330127, -6)--(-4.330127, 7);
\draw [gray,very thin] (-3.464102, -6)--(-3.464102, 7);
\draw [gray,very thin] (-2.598076, -6)--(-2.598076, 7);
\draw [gray,very thin] (-1.732051, -6)--(-1.732051, 7);
\draw [gray,very thin] (-0.8660254, -6)--(-0.8660254, 7);
\draw [gray,very thin] (0, -6)--(0, 7);
\draw [gray,very thin] (0.8660254, -6)--(0.8660254, 7);
\draw [gray,very thin] (-6.928203, -5.5)--(1.732051, -5.5);
\draw [gray,very thin] (-6.928203, -5)--(1.732051, -5);
\draw [gray,very thin] (-6.928203, -4.5)--(1.732051, -4.5);
\draw [gray,very thin] (-6.928203, -4)--(1.732051, -4);
\draw [gray,very thin] (-6.928203, -3.5)--(1.732051, -3.5);
\draw [gray,very thin] (-6.928203, -3)--(1.732051, -3);
\draw [gray,very thin] (-6.928203, -2.5)--(1.732051, -2.5);
\draw [gray,very thin] (-6.928203, -2)--(1.732051, -2);
\draw [gray,very thin] (-6.928203, -1.5)--(1.732051, -1.5);
\draw [gray,very thin] (-6.928203, -1)--(1.732051, -1);
\draw [gray,very thin] (-6.928203, -0.5)--(1.732051, -0.5);
\draw [gray,very thin] (-6.928203, 0)--(1.732051, 0);
\draw [gray,very thin] (-6.928203, 0.5)--(1.732051, 0.5);
\draw [gray,very thin] (-6.928203, 1)--(1.732051, 1);
\draw [gray,very thin] (-6.928203, 1.5)--(1.732051, 1.5);
\draw [gray,very thin] (-6.928203, 2)--(1.732051, 2);
\draw [gray,very thin] (-6.928203, 2.5)--(1.732051, 2.5);
\draw [gray,very thin] (-6.928203, 3)--(1.732051, 3);
\draw [gray,very thin] (-6.928203, 3.5)--(1.732051, 3.5);
\draw [gray,very thin] (-6.928203, 4)--(1.732051, 4);
\draw [gray,very thin] (-6.928203, 4.5)--(1.732051, 4.5);
\draw [gray,very thin] (-6.928203, 5)--(1.732051, 5);
\draw [gray,very thin] (-6.928203, 5.5)--(1.732051, 5.5);
\draw [gray,very thin] (-6.928203, 6)--(1.732051, 6);
\draw [gray,very thin] (-6.928203, 6.5)--(1.732051, 6.5);
\draw [very thick] (-0.8660254, -0.5)--(0, 0);
\draw [very thick] (0.8660254, -0.5)--(0, 0);
\draw [very thick] (0, 1)--(0, 0);
\draw [very thick] (-1.732051, 0)--(-0.8660254, -0.5);
\draw [very thick] (-0.8660254, -1.5)--(-0.8660254, -0.5);
\draw [very thick] (1.732051, 0)--(0.8660254, -0.5);
\draw [very thick] (0.8660254, -1.5)--(0.8660254, -0.5);
\draw [very thick] (-0.8660254, 1.5)--(0, 1);
\draw [very thick] (0.8660254, 1.5)--(0, 1);
\draw [very thick] (-2.598076, -0.5)--(-1.732051, 0);
\draw [very thick] (-1.732051, 1)--(-0.8660254, 1.5);
\draw [very thick] (-1.732051, 1)--(-1.732051, 0);
\draw [very thick] (-1.732051, -2)--(-0.8660254, -1.5);
\draw [very thick] (0, -2)--(-0.8660254, -1.5);
\draw [very thick] (0, -2)--(0.8660254, -1.5);
\draw [very thick] (1.732051, 1)--(0.8660254, 1.5);
\draw [very thick] (1.732051, -2)--(0.8660254, -1.5);
\draw [very thick] (-0.8660254, 2.5)--(-0.8660254, 1.5);
\draw [very thick] (0.8660254, 2.5)--(0.8660254, 1.5);
\draw [very thick] (-3.464102, 0)--(-2.598076, -0.5);
\draw [very thick] (-2.598076, -1.5)--(-1.732051, -2);
\draw [very thick] (-2.598076, -1.5)--(-2.598076, -0.5);
\draw [very thick] (-2.598076, 1.5)--(-1.732051, 1);
\draw [very thick] (-1.732051, -3)--(-1.732051, -2);
\draw [very thick] (0, -3)--(0, -2);
\draw [very thick] (-1.732051, 3)--(-0.8660254, 2.5);
\draw [very thick] (0, 3)--(-0.8660254, 2.5);
\draw [very thick] (0, 3)--(0.8660254, 2.5);
\draw [very thick] (1.732051, 3)--(0.8660254, 2.5);
\draw [very thick] (-4.330127, -0.5)--(-3.464102, 0);
\draw [very thick] (-3.464102, 1)--(-2.598076, 1.5);
\draw [very thick] (-3.464102, 1)--(-3.464102, 0);
\draw [very thick] (-3.464102, -2)--(-2.598076, -1.5);
\draw [very thick] (-2.598076, 2.5)--(-1.732051, 3);
\draw [very thick] (-2.598076, 2.5)--(-2.598076, 1.5);
\draw [very thick] (-2.598076, -3.5)--(-1.732051, -3);
\draw [very thick] (-0.8660254, -3.5)--(-1.732051, -3);
\draw [very thick] (-0.8660254, -3.5)--(0, -3);
\draw [very thick] (0.8660254, -3.5)--(0, -3);
\draw [very thick] (0.8660254, -3.5)--(1.732051, -3);
\draw [very thick] (-1.732051, 4)--(-1.732051, 3);
\draw [very thick] (0, 4)--(0, 3);
\draw [very thick] (-5.196152, 0)--(-4.330127, -0.5);
\draw [very thick] (-4.330127, -1.5)--(-3.464102, -2);
\draw [very thick] (-4.330127, -1.5)--(-4.330127, -0.5);
\draw [very thick] (-4.330127, 1.5)--(-3.464102, 1);
\draw [very thick] (-3.464102, -3)--(-2.598076, -3.5);
\draw [very thick] (-3.464102, -3)--(-3.464102, -2);
\draw [very thick] (-3.464102, 3)--(-2.598076, 2.5);
\draw [very thick] (-2.598076, -4.5)--(-2.598076, -3.5);
\draw [very thick] (-0.8660254, -4.5)--(-0.8660254, -3.5);
\draw [very thick] (0.8660254, -4.5)--(0.8660254, -3.5);
\draw [very thick] (-2.598076, 4.5)--(-1.732051, 4);
\draw [very thick] (-0.8660254, 4.5)--(-1.732051, 4);
\draw [very thick] (-0.8660254, 4.5)--(0, 4);
\draw [very thick] (0.8660254, 4.5)--(0, 4);
\draw [very thick] (0.8660254, 4.5)--(1.732051, 4);
\draw [very thick] (-6.062178, -0.5)--(-5.196152, 0);
\draw [very thick] (-5.196152, 1)--(-4.330127, 1.5);
\draw [very thick] (-5.196152, 1)--(-5.196152, 0);
\draw [very thick] (-5.196152, -2)--(-4.330127, -1.5);
\draw [very thick] (-4.330127, 2.5)--(-3.464102, 3);
\draw [very thick] (-4.330127, 2.5)--(-4.330127, 1.5);
\draw [very thick] (-4.330127, -3.5)--(-3.464102, -3);
\draw [very thick] (-3.464102, 4)--(-2.598076, 4.5);
\draw [very thick] (-3.464102, 4)--(-3.464102, 3);
\draw [very thick] (-3.464102, -5)--(-2.598076, -4.5);
\draw [very thick] (-1.732051, -5)--(-2.598076, -4.5);
\draw [very thick] (-1.732051, -5)--(-0.8660254, -4.5);
\draw [very thick] (0, -5)--(-0.8660254, -4.5);
\draw [very thick] (0, -5)--(0.8660254, -4.5);
\draw [very thick] (1.732051, -5)--(0.8660254, -4.5);
\draw [very thick] (-2.598076, 5.5)--(-2.598076, 4.5);
\draw [very thick] (-0.8660254, 5.5)--(-0.8660254, 4.5);
\draw [very thick] (0.8660254, 5.5)--(0.8660254, 4.5);
\draw [very thick] (-6.928203, 0)--(-6.062178, -0.5);
\draw [very thick] (-6.062178, -1.5)--(-5.196152, -2);
\draw [very thick] (-6.062178, -1.5)--(-6.062178, -0.5);
\draw [very thick] (-6.062178, 1.5)--(-5.196152, 1);
\draw [very thick] (-5.196152, -3)--(-4.330127, -3.5);
\draw [very thick] (-5.196152, -3)--(-5.196152, -2);
\draw [very thick] (-5.196152, 3)--(-4.330127, 2.5);
\draw [very thick] (-4.330127, -4.5)--(-3.464102, -5);
\draw [very thick] (-4.330127, -4.5)--(-4.330127, -3.5);
\draw [very thick] (-4.330127, 4.5)--(-3.464102, 4);
\draw [very thick] (-3.464102, -6)--(-3.464102, -5);
\draw [very thick] (-1.732051, -6)--(-1.732051, -5);
\draw [very thick] (0, -6)--(0, -5);
\draw [very thick] (-3.464102, 6)--(-2.598076, 5.5);
\draw [very thick] (-1.732051, 6)--(-2.598076, 5.5);
\draw [very thick] (-1.732051, 6)--(-0.8660254, 5.5);
\draw [very thick] (0, 6)--(-0.8660254, 5.5);
\draw [very thick] (0, 6)--(0.8660254, 5.5);
\draw [very thick] (1.732051, 6)--(0.8660254, 5.5);
\draw [very thick] (-6.928203, 1)--(-6.062178, 1.5);
\draw [very thick] (-6.928203, -2)--(-6.062178, -1.5);
\draw [very thick] (-6.062178, 2.5)--(-5.196152, 3);
\draw [very thick] (-6.062178, 2.5)--(-6.062178, 1.5);
\draw [very thick] (-6.062178, -3.5)--(-5.196152, -3);
\draw [very thick] (-5.196152, 4)--(-4.330127, 4.5);
\draw [very thick] (-5.196152, 4)--(-5.196152, 3);
\draw [very thick] (-5.196152, -5)--(-4.330127, -4.5);
\draw [very thick] (-4.330127, 5.5)--(-3.464102, 6);
\draw [very thick] (-4.330127, 5.5)--(-4.330127, 4.5);
\draw [very thick] (-3.464102, 7)--(-3.464102, 6);
\draw [very thick] (-1.732051, 7)--(-1.732051, 6);
\draw [very thick] (0, 7)--(0, 6);
\draw [very thick] (-6.928203, -3)--(-6.062178, -3.5);
\draw [very thick] (-6.928203, 3)--(-6.062178, 2.5);
\draw [very thick] (-6.062178, -4.5)--(-5.196152, -5);
\draw [very thick] (-6.062178, -4.5)--(-6.062178, -3.5);
\draw [very thick] (-6.062178, 4.5)--(-5.196152, 4);
\draw [very thick] (-5.196152, -6)--(-5.196152, -5);
\draw [very thick] (-5.196152, 6)--(-4.330127, 5.5);
\draw [very thick] (-6.928203, 4)--(-6.062178, 4.5);
\draw [very thick] (-6.928203, -5)--(-6.062178, -4.5);
\draw [very thick] (-6.062178, 5.5)--(-5.196152, 6);
\draw [very thick] (-6.062178, 5.5)--(-6.062178, 4.5);
\draw [very thick] (-5.196152, 7)--(-5.196152, 6);
\draw [very thick] (-6.928203, 6)--(-6.062178, 5.5);
\draw [black,fill=red,rotate around={45:(0, 0)}] plot[mark=square*,mark size=5.4pt] (0, 0);
\draw [black,fill=smoothred,rotate around={45:(-0.8660254, -0.5)}] plot[mark=square*,mark size=5.4pt] (-0.8660254, -0.5);
\draw [black,fill=smoothgreen] plot[mark=square*,mark size=5.4pt] (0.8660254, -0.5);
\draw [black,fill=smoothred,rotate around={45:(0, 1)}] plot[mark=square*,mark size=5.4pt] (0, 1);
\draw [black,fill=smoothred,rotate around={45:(-1.732051, 0)}] plot[mark=square*,mark size=5.4pt] (-1.732051, 0);
\draw [black,fill=smoothred,rotate around={45:(-0.8660254, -1.5)}] plot[mark=square*,mark size=5.4pt] (-0.8660254, -1.5);
\draw [fill=lightgray] (0.8660254, -1.5) circle (5.4pt);
\draw [black,fill=smoothred,rotate around={45:(-0.8660254, 1.5)}] plot[mark=square*,mark size=5.4pt] (-0.8660254, 1.5);
\draw [black,fill=smoothgreen] plot[mark=square*,mark size=5.4pt] (0.8660254, 1.5);
\draw [black,fill=smoothred,rotate around={45:(-2.598076, -0.5)}] plot[mark=square*,mark size=5.4pt] (-2.598076, -0.5);
\draw [black,fill=smoothred,rotate around={45:(-1.732051, 1)}] plot[mark=square*,mark size=5.4pt] (-1.732051, 1);
\draw [black,fill=smoothred,rotate around={45:(-1.732051, -2)}] plot[mark=square*,mark size=5.4pt] (-1.732051, -2);
\draw [black,fill=smoothgreen] plot[mark=square*,mark size=5.4pt] (0, -2);
\draw [black,fill=smoothred,rotate around={45:(-0.8660254, 2.5)}] plot[mark=square*,mark size=5.4pt] (-0.8660254, 2.5);
\draw [fill=lightgray] (0.8660254, 2.5) circle (5.4pt);
\draw [black,fill=smoothred,rotate around={45:(-3.464102, 0)}] plot[mark=square*,mark size=5.4pt] (-3.464102, 0);
\draw [black,fill=smoothred,rotate around={45:(-2.598076, -1.5)}] plot[mark=square*,mark size=5.4pt] (-2.598076, -1.5);
\draw [black,fill=smoothred,rotate around={45:(-2.598076, 1.5)}] plot[mark=square*,mark size=5.4pt] (-2.598076, 1.5);
\draw [black,fill=smoothred,rotate around={45:(-1.732051, -3)}] plot[mark=square*,mark size=5.4pt] (-1.732051, -3);
\draw [fill=lightgray] (0, -3) circle (5.4pt);
\draw [black,fill=smoothred,rotate around={45:(-1.732051, 3)}] plot[mark=square*,mark size=5.4pt] (-1.732051, 3);
\draw [black,fill=smoothgreen] plot[mark=square*,mark size=5.4pt] (0, 3);
\draw [black,fill=smoothred,rotate around={45:(-4.330127, -0.5)}] plot[mark=square*,mark size=5.4pt] (-4.330127, -0.5);
\draw [black,fill=smoothred,rotate around={45:(-3.464102, 1)}] plot[mark=square*,mark size=5.4pt] (-3.464102, 1);
\draw [black,fill=smoothred,rotate around={45:(-3.464102, -2)}] plot[mark=square*,mark size=5.4pt] (-3.464102, -2);
\draw [black,fill=smoothred,rotate around={45:(-2.598076, 2.5)}] plot[mark=square*,mark size=5.4pt] (-2.598076, 2.5);
\draw [black,fill=smoothred,rotate around={45:(-2.598076, -3.5)}] plot[mark=square*,mark size=5.4pt] (-2.598076, -3.5);
\draw [black,fill=smoothgreen] plot[mark=square*,mark size=5.4pt] (-0.8660254, -3.5);
\draw [fill=lightgray] (0.8660254, -3.5) circle (5.4pt);
\draw [black,fill=smoothred,rotate around={45:(-1.732051, 4)}] plot[mark=square*,mark size=5.4pt] (-1.732051, 4);
\draw [fill=lightgray] (0, 4) circle (5.4pt);
\draw [black,fill=smoothred,rotate around={45:(-5.196152, 0)}] plot[mark=square*,mark size=5.4pt] (-5.196152, 0);
\draw [black,fill=smoothred,rotate around={45:(-4.330127, -1.5)}] plot[mark=square*,mark size=5.4pt] (-4.330127, -1.5);
\draw [black,fill=smoothred,rotate around={45:(-4.330127, 1.5)}] plot[mark=square*,mark size=5.4pt] (-4.330127, 1.5);
\draw [black,fill=smoothred,rotate around={45:(-3.464102, -3)}] plot[mark=square*,mark size=5.4pt] (-3.464102, -3);
\draw [black,fill=smoothred,rotate around={45:(-3.464102, 3)}] plot[mark=square*,mark size=5.4pt] (-3.464102, 3);
\draw [black,fill=smoothgreen] plot[mark=square*,mark size=5.4pt] (-2.598076, -4.5);
\draw [fill=lightgray] (-0.8660254, -4.5) circle (5.4pt);
\draw [fill=lightgray] (0.8660254, -4.5) circle (5.4pt);
\draw [black,fill=smoothred,rotate around={45:(-2.598076, 4.5)}] plot[mark=square*,mark size=5.4pt] (-2.598076, 4.5);
\draw [black,fill=smoothgreen] plot[mark=square*,mark size=5.4pt] (-0.8660254, 4.5);
\draw [fill=lightgray] (0.8660254, 4.5) circle (5.4pt);
\draw [black,fill=smoothred,rotate around={45:(-6.062178, -0.5)}] plot[mark=square*,mark size=5.4pt] (-6.062178, -0.5);
\draw [black,fill=smoothred,rotate around={45:(-5.196152, 1)}] plot[mark=square*,mark size=5.4pt] (-5.196152, 1);
\draw [black,fill=smoothred,rotate around={45:(-5.196152, -2)}] plot[mark=square*,mark size=5.4pt] (-5.196152, -2);
\draw [black,fill=smoothred,rotate around={45:(-4.330127, 2.5)}] plot[mark=square*,mark size=5.4pt] (-4.330127, 2.5);
\draw [black,fill=smoothred,rotate around={45:(-4.330127, -3.5)}] plot[mark=square*,mark size=5.4pt] (-4.330127, -3.5);
\draw [black,fill=smoothred,rotate around={45:(-3.464102, 4)}] plot[mark=square*,mark size=5.4pt] (-3.464102, 4);
\draw [fill=lightgray] (-3.464102, -5) circle (5.4pt);
\draw [fill=lightgray] (-1.732051, -5) circle (5.4pt);
\draw [fill=lightgray] (0, -5) circle (5.4pt);
\draw [black,fill=smoothgreen] plot[mark=square*,mark size=5.4pt] (-2.598076, 5.5);
\draw [fill=lightgray] (-0.8660254, 5.5) circle (5.4pt);
\draw [fill=lightgray] (0.8660254, 5.5) circle (5.4pt);
\draw [fill=lightgray] (-6.062178, -1.5) circle (5.4pt);
\draw [fill=lightgray] (-6.062178, 1.5) circle (5.4pt);
\draw [fill=lightgray] (-5.196152, -3) circle (5.4pt);
\draw [fill=lightgray] (-5.196152, 3) circle (5.4pt);
\draw [fill=lightgray] (-4.330127, -4.5) circle (5.4pt);
\draw [fill=lightgray] (-4.330127, 4.5) circle (5.4pt);
\draw [fill=lightgray] (-3.464102, 6) circle (5.4pt);
\draw [fill=lightgray] (-1.732051, 6) circle (5.4pt);
\draw [fill=lightgray] (0, 6) circle (5.4pt);
\draw [fill=lightgray] (-6.062178, 2.5) circle (5.4pt);
\draw [fill=lightgray] (-6.062178, -3.5) circle (5.4pt);
\draw [fill=lightgray] (-5.196152, 4) circle (5.4pt);
\draw [fill=lightgray] (-5.196152, -5) circle (5.4pt);
\draw [fill=lightgray] (-4.330127, 5.5) circle (5.4pt);
\draw [fill=lightgray] (-6.062178, -4.5) circle (5.4pt);
\draw [fill=lightgray] (-6.062178, 4.5) circle (5.4pt);
\draw [fill=lightgray] (-5.196152, 6) circle (5.4pt);
\draw [fill=lightgray] (-6.062178, 5.5) circle (5.4pt);
\draw [red,fill opacity=0,line width=1.8pt] (0, 0) circle (14.4pt);
\end{tikzpicture}
\caption{Turn $2\tau + 5$.}
\label{Bent Ray}
\end{subfigure}
\hspace{0.05\textwidth}
\hspace*{-1em}
\begin{subfigure}{0.24285714285714283\textwidth}
\begin{tikzpicture}[scale=0.4]
\definecolor{lightgray}{RGB}{195,195,195}
\definecolor{smoothred}{RGB}{255,150,150};
\definecolor{smoothgreen}{RGB}{0,105,105};
\draw [gray,very thin] (-7.794229, -6)--(-7.794229, 7);
\draw [gray,very thin] (-6.928203, -6)--(-6.928203, 7);
\draw [gray,very thin] (-6.062178, -6)--(-6.062178, 7);
\draw [gray,very thin] (-5.196152, -6)--(-5.196152, 7);
\draw [gray,very thin] (-4.330127, -6)--(-4.330127, 7);
\draw [gray,very thin] (-3.464102, -6)--(-3.464102, 7);
\draw [gray,very thin] (-2.598076, -6)--(-2.598076, 7);
\draw [gray,very thin] (-1.732051, -6)--(-1.732051, 7);
\draw [gray,very thin] (-0.8660254, -6)--(-0.8660254, 7);
\draw [gray,very thin] (0, -6)--(0, 7);
\draw [gray,very thin] (0.8660254, -6)--(0.8660254, 7);
\draw [gray,very thin] (-8.660254, -5.5)--(1.732051, -5.5);
\draw [gray,very thin] (-8.660254, -5)--(1.732051, -5);
\draw [gray,very thin] (-8.660254, -4.5)--(1.732051, -4.5);
\draw [gray,very thin] (-8.660254, -4)--(1.732051, -4);
\draw [gray,very thin] (-8.660254, -3.5)--(1.732051, -3.5);
\draw [gray,very thin] (-8.660254, -3)--(1.732051, -3);
\draw [gray,very thin] (-8.660254, -2.5)--(1.732051, -2.5);
\draw [gray,very thin] (-8.660254, -2)--(1.732051, -2);
\draw [gray,very thin] (-8.660254, -1.5)--(1.732051, -1.5);
\draw [gray,very thin] (-8.660254, -1)--(1.732051, -1);
\draw [gray,very thin] (-8.660254, -0.5)--(1.732051, -0.5);
\draw [gray,very thin] (-8.660254, 0)--(1.732051, 0);
\draw [gray,very thin] (-8.660254, 0.5)--(1.732051, 0.5);
\draw [gray,very thin] (-8.660254, 1)--(1.732051, 1);
\draw [gray,very thin] (-8.660254, 1.5)--(1.732051, 1.5);
\draw [gray,very thin] (-8.660254, 2)--(1.732051, 2);
\draw [gray,very thin] (-8.660254, 2.5)--(1.732051, 2.5);
\draw [gray,very thin] (-8.660254, 3)--(1.732051, 3);
\draw [gray,very thin] (-8.660254, 3.5)--(1.732051, 3.5);
\draw [gray,very thin] (-8.660254, 4)--(1.732051, 4);
\draw [gray,very thin] (-8.660254, 4.5)--(1.732051, 4.5);
\draw [gray,very thin] (-8.660254, 5)--(1.732051, 5);
\draw [gray,very thin] (-8.660254, 5.5)--(1.732051, 5.5);
\draw [gray,very thin] (-8.660254, 6)--(1.732051, 6);
\draw [gray,very thin] (-8.660254, 6.5)--(1.732051, 6.5);
\draw [very thick] (-0.8660254, -0.5)--(0, 0);
\draw [very thick] (0.8660254, -0.5)--(0, 0);
\draw [very thick] (0, 1)--(0, 0);
\draw [very thick] (-1.732051, 0)--(-0.8660254, -0.5);
\draw [very thick] (-0.8660254, -1.5)--(-0.8660254, -0.5);
\draw [very thick] (1.732051, 0)--(0.8660254, -0.5);
\draw [very thick] (0.8660254, -1.5)--(0.8660254, -0.5);
\draw [very thick] (-0.8660254, 1.5)--(0, 1);
\draw [very thick] (0.8660254, 1.5)--(0, 1);
\draw [very thick] (-2.598076, -0.5)--(-1.732051, 0);
\draw [very thick] (-1.732051, 1)--(-0.8660254, 1.5);
\draw [very thick] (-1.732051, 1)--(-1.732051, 0);
\draw [very thick] (-1.732051, -2)--(-0.8660254, -1.5);
\draw [very thick] (0, -2)--(-0.8660254, -1.5);
\draw [very thick] (0, -2)--(0.8660254, -1.5);
\draw [very thick] (1.732051, 1)--(0.8660254, 1.5);
\draw [very thick] (1.732051, -2)--(0.8660254, -1.5);
\draw [very thick] (-0.8660254, 2.5)--(-0.8660254, 1.5);
\draw [very thick] (0.8660254, 2.5)--(0.8660254, 1.5);
\draw [very thick] (-3.464102, 0)--(-2.598076, -0.5);
\draw [very thick] (-2.598076, -1.5)--(-1.732051, -2);
\draw [very thick] (-2.598076, -1.5)--(-2.598076, -0.5);
\draw [very thick] (-2.598076, 1.5)--(-1.732051, 1);
\draw [very thick] (-1.732051, -3)--(-1.732051, -2);
\draw [very thick] (0, -3)--(0, -2);
\draw [very thick] (-1.732051, 3)--(-0.8660254, 2.5);
\draw [very thick] (0, 3)--(-0.8660254, 2.5);
\draw [very thick] (0, 3)--(0.8660254, 2.5);
\draw [very thick] (1.732051, 3)--(0.8660254, 2.5);
\draw [very thick] (-4.330127, -0.5)--(-3.464102, 0);
\draw [very thick] (-3.464102, 1)--(-2.598076, 1.5);
\draw [very thick] (-3.464102, 1)--(-3.464102, 0);
\draw [very thick] (-3.464102, -2)--(-2.598076, -1.5);
\draw [very thick] (-2.598076, 2.5)--(-1.732051, 3);
\draw [very thick] (-2.598076, 2.5)--(-2.598076, 1.5);
\draw [very thick] (-2.598076, -3.5)--(-1.732051, -3);
\draw [very thick] (-0.8660254, -3.5)--(-1.732051, -3);
\draw [very thick] (-0.8660254, -3.5)--(0, -3);
\draw [very thick] (0.8660254, -3.5)--(0, -3);
\draw [very thick] (0.8660254, -3.5)--(1.732051, -3);
\draw [very thick] (-1.732051, 4)--(-1.732051, 3);
\draw [very thick] (0, 4)--(0, 3);
\draw [very thick] (-5.196152, 0)--(-4.330127, -0.5);
\draw [very thick] (-4.330127, -1.5)--(-3.464102, -2);
\draw [very thick] (-4.330127, -1.5)--(-4.330127, -0.5);
\draw [very thick] (-4.330127, 1.5)--(-3.464102, 1);
\draw [very thick] (-3.464102, -3)--(-2.598076, -3.5);
\draw [very thick] (-3.464102, -3)--(-3.464102, -2);
\draw [very thick] (-3.464102, 3)--(-2.598076, 2.5);
\draw [very thick] (-2.598076, -4.5)--(-2.598076, -3.5);
\draw [very thick] (-0.8660254, -4.5)--(-0.8660254, -3.5);
\draw [very thick] (0.8660254, -4.5)--(0.8660254, -3.5);
\draw [very thick] (-2.598076, 4.5)--(-1.732051, 4);
\draw [very thick] (-0.8660254, 4.5)--(-1.732051, 4);
\draw [very thick] (-0.8660254, 4.5)--(0, 4);
\draw [very thick] (0.8660254, 4.5)--(0, 4);
\draw [very thick] (0.8660254, 4.5)--(1.732051, 4);
\draw [very thick] (-6.062178, -0.5)--(-5.196152, 0);
\draw [very thick] (-5.196152, 1)--(-4.330127, 1.5);
\draw [very thick] (-5.196152, 1)--(-5.196152, 0);
\draw [very thick] (-5.196152, -2)--(-4.330127, -1.5);
\draw [very thick] (-4.330127, 2.5)--(-3.464102, 3);
\draw [very thick] (-4.330127, 2.5)--(-4.330127, 1.5);
\draw [very thick] (-4.330127, -3.5)--(-3.464102, -3);
\draw [very thick] (-3.464102, 4)--(-2.598076, 4.5);
\draw [very thick] (-3.464102, 4)--(-3.464102, 3);
\draw [very thick] (-3.464102, -5)--(-2.598076, -4.5);
\draw [very thick] (-1.732051, -5)--(-2.598076, -4.5);
\draw [very thick] (-1.732051, -5)--(-0.8660254, -4.5);
\draw [very thick] (0, -5)--(-0.8660254, -4.5);
\draw [very thick] (0, -5)--(0.8660254, -4.5);
\draw [very thick] (1.732051, -5)--(0.8660254, -4.5);
\draw [very thick] (-2.598076, 5.5)--(-2.598076, 4.5);
\draw [very thick] (-0.8660254, 5.5)--(-0.8660254, 4.5);
\draw [very thick] (0.8660254, 5.5)--(0.8660254, 4.5);
\draw [very thick] (-6.928203, 0)--(-6.062178, -0.5);
\draw [very thick] (-6.062178, -1.5)--(-5.196152, -2);
\draw [very thick] (-6.062178, -1.5)--(-6.062178, -0.5);
\draw [very thick] (-6.062178, 1.5)--(-5.196152, 1);
\draw [very thick] (-5.196152, -3)--(-4.330127, -3.5);
\draw [very thick] (-5.196152, -3)--(-5.196152, -2);
\draw [very thick] (-5.196152, 3)--(-4.330127, 2.5);
\draw [very thick] (-4.330127, -4.5)--(-3.464102, -5);
\draw [very thick] (-4.330127, -4.5)--(-4.330127, -3.5);
\draw [very thick] (-4.330127, 4.5)--(-3.464102, 4);
\draw [very thick] (-3.464102, -6)--(-3.464102, -5);
\draw [very thick] (-1.732051, -6)--(-1.732051, -5);
\draw [very thick] (0, -6)--(0, -5);
\draw [very thick] (-3.464102, 6)--(-2.598076, 5.5);
\draw [very thick] (-1.732051, 6)--(-2.598076, 5.5);
\draw [very thick] (-1.732051, 6)--(-0.8660254, 5.5);
\draw [very thick] (0, 6)--(-0.8660254, 5.5);
\draw [very thick] (0, 6)--(0.8660254, 5.5);
\draw [very thick] (1.732051, 6)--(0.8660254, 5.5);
\draw [very thick] (-7.794229, -0.5)--(-6.928203, 0);
\draw [very thick] (-6.928203, 1)--(-6.062178, 1.5);
\draw [very thick] (-6.928203, 1)--(-6.928203, 0);
\draw [very thick] (-6.928203, -2)--(-6.062178, -1.5);
\draw [very thick] (-6.062178, 2.5)--(-5.196152, 3);
\draw [very thick] (-6.062178, 2.5)--(-6.062178, 1.5);
\draw [very thick] (-6.062178, -3.5)--(-5.196152, -3);
\draw [very thick] (-5.196152, 4)--(-4.330127, 4.5);
\draw [very thick] (-5.196152, 4)--(-5.196152, 3);
\draw [very thick] (-5.196152, -5)--(-4.330127, -4.5);
\draw [very thick] (-4.330127, 5.5)--(-3.464102, 6);
\draw [very thick] (-4.330127, 5.5)--(-4.330127, 4.5);
\draw [very thick] (-3.464102, 7)--(-3.464102, 6);
\draw [very thick] (-1.732051, 7)--(-1.732051, 6);
\draw [very thick] (0, 7)--(0, 6);
\draw [very thick] (-8.660254, 0)--(-7.794229, -0.5);
\draw [very thick] (-7.794229, -1.5)--(-6.928203, -2);
\draw [very thick] (-7.794229, -1.5)--(-7.794229, -0.5);
\draw [very thick] (-7.794229, 1.5)--(-6.928203, 1);
\draw [very thick] (-6.928203, -3)--(-6.062178, -3.5);
\draw [very thick] (-6.928203, -3)--(-6.928203, -2);
\draw [very thick] (-6.928203, 3)--(-6.062178, 2.5);
\draw [very thick] (-6.062178, -4.5)--(-5.196152, -5);
\draw [very thick] (-6.062178, -4.5)--(-6.062178, -3.5);
\draw [very thick] (-6.062178, 4.5)--(-5.196152, 4);
\draw [very thick] (-5.196152, -6)--(-5.196152, -5);
\draw [very thick] (-5.196152, 6)--(-4.330127, 5.5);
\draw [very thick] (-8.660254, 1)--(-7.794229, 1.5);
\draw [very thick] (-8.660254, -2)--(-7.794229, -1.5);
\draw [very thick] (-7.794229, 2.5)--(-6.928203, 3);
\draw [very thick] (-7.794229, 2.5)--(-7.794229, 1.5);
\draw [very thick] (-7.794229, -3.5)--(-6.928203, -3);
\draw [very thick] (-6.928203, 4)--(-6.062178, 4.5);
\draw [very thick] (-6.928203, 4)--(-6.928203, 3);
\draw [very thick] (-6.928203, -5)--(-6.062178, -4.5);
\draw [very thick] (-6.062178, 5.5)--(-5.196152, 6);
\draw [very thick] (-6.062178, 5.5)--(-6.062178, 4.5);
\draw [very thick] (-5.196152, 7)--(-5.196152, 6);
\draw [very thick] (-8.660254, -3)--(-7.794229, -3.5);
\draw [very thick] (-8.660254, 3)--(-7.794229, 2.5);
\draw [very thick] (-7.794229, -4.5)--(-6.928203, -5);
\draw [very thick] (-7.794229, -4.5)--(-7.794229, -3.5);
\draw [very thick] (-7.794229, 4.5)--(-6.928203, 4);
\draw [very thick] (-6.928203, -6)--(-6.928203, -5);
\draw [very thick] (-6.928203, 6)--(-6.062178, 5.5);
\draw [very thick] (-8.660254, 4)--(-7.794229, 4.5);
\draw [very thick] (-8.660254, -5)--(-7.794229, -4.5);
\draw [very thick] (-7.794229, 5.5)--(-6.928203, 6);
\draw [very thick] (-7.794229, 5.5)--(-7.794229, 4.5);
\draw [very thick] (-6.928203, 7)--(-6.928203, 6);
\draw [very thick] (-8.660254, 6)--(-7.794229, 5.5);
\draw [black,fill=red,rotate around={45:(0, 0)}] plot[mark=square*,mark size=5.4pt] (0, 0);
\draw [black,fill=smoothred,rotate around={45:(-0.8660254, -0.5)}] plot[mark=square*,mark size=5.4pt] (-0.8660254, -0.5);
\draw [black,fill=smoothgreen] plot[mark=square*,mark size=5.4pt] (0.8660254, -0.5);
\draw [black,fill=smoothred,rotate around={45:(0, 1)}] plot[mark=square*,mark size=5.4pt] (0, 1);
\draw [black,fill=smoothred,rotate around={45:(-1.732051, 0)}] plot[mark=square*,mark size=5.4pt] (-1.732051, 0);
\draw [black,fill=smoothred,rotate around={45:(-0.8660254, -1.5)}] plot[mark=square*,mark size=5.4pt] (-0.8660254, -1.5);
\draw [fill=lightgray] (0.8660254, -1.5) circle (5.4pt);
\draw [black,fill=smoothred,rotate around={45:(-0.8660254, 1.5)}] plot[mark=square*,mark size=5.4pt] (-0.8660254, 1.5);
\draw [black,fill=smoothgreen] plot[mark=square*,mark size=5.4pt] (0.8660254, 1.5);
\draw [black,fill=smoothred,rotate around={45:(-2.598076, -0.5)}] plot[mark=square*,mark size=5.4pt] (-2.598076, -0.5);
\draw [black,fill=smoothred,rotate around={45:(-1.732051, 1)}] plot[mark=square*,mark size=5.4pt] (-1.732051, 1);
\draw [black,fill=smoothred,rotate around={45:(-1.732051, -2)}] plot[mark=square*,mark size=5.4pt] (-1.732051, -2);
\draw [black,fill=smoothgreen] plot[mark=square*,mark size=5.4pt] (0, -2);
\draw [black,fill=smoothred,rotate around={45:(-0.8660254, 2.5)}] plot[mark=square*,mark size=5.4pt] (-0.8660254, 2.5);
\draw [fill=lightgray] (0.8660254, 2.5) circle (5.4pt);
\draw [black,fill=smoothred,rotate around={45:(-3.464102, 0)}] plot[mark=square*,mark size=5.4pt] (-3.464102, 0);
\draw [black,fill=smoothred,rotate around={45:(-2.598076, -1.5)}] plot[mark=square*,mark size=5.4pt] (-2.598076, -1.5);
\draw [black,fill=smoothred,rotate around={45:(-2.598076, 1.5)}] plot[mark=square*,mark size=5.4pt] (-2.598076, 1.5);
\draw [black,fill=smoothred,rotate around={45:(-1.732051, -3)}] plot[mark=square*,mark size=5.4pt] (-1.732051, -3);
\draw [fill=lightgray] (0, -3) circle (5.4pt);
\draw [black,fill=smoothred,rotate around={45:(-1.732051, 3)}] plot[mark=square*,mark size=5.4pt] (-1.732051, 3);
\draw [black,fill=smoothgreen] plot[mark=square*,mark size=5.4pt] (0, 3);
\draw [black,fill=smoothred,rotate around={45:(-4.330127, -0.5)}] plot[mark=square*,mark size=5.4pt] (-4.330127, -0.5);
\draw [black,fill=smoothred,rotate around={45:(-3.464102, 1)}] plot[mark=square*,mark size=5.4pt] (-3.464102, 1);
\draw [black,fill=smoothred,rotate around={45:(-3.464102, -2)}] plot[mark=square*,mark size=5.4pt] (-3.464102, -2);
\draw [black,fill=smoothred,rotate around={45:(-2.598076, 2.5)}] plot[mark=square*,mark size=5.4pt] (-2.598076, 2.5);
\draw [black,fill=smoothred,rotate around={45:(-2.598076, -3.5)}] plot[mark=square*,mark size=5.4pt] (-2.598076, -3.5);
\draw [black,fill=smoothgreen] plot[mark=square*,mark size=5.4pt] (-0.8660254, -3.5);
\draw [fill=lightgray] (0.8660254, -3.5) circle (5.4pt);
\draw [black,fill=smoothred,rotate around={45:(-1.732051, 4)}] plot[mark=square*,mark size=5.4pt] (-1.732051, 4);
\draw [fill=lightgray] (0, 4) circle (5.4pt);
\draw [black,fill=smoothred,rotate around={45:(-5.196152, 0)}] plot[mark=square*,mark size=5.4pt] (-5.196152, 0);
\draw [black,fill=smoothred,rotate around={45:(-4.330127, -1.5)}] plot[mark=square*,mark size=5.4pt] (-4.330127, -1.5);
\draw [black,fill=smoothred,rotate around={45:(-4.330127, 1.5)}] plot[mark=square*,mark size=5.4pt] (-4.330127, 1.5);
\draw [black,fill=smoothred,rotate around={45:(-3.464102, -3)}] plot[mark=square*,mark size=5.4pt] (-3.464102, -3);
\draw [black,fill=smoothred,rotate around={45:(-3.464102, 3)}] plot[mark=square*,mark size=5.4pt] (-3.464102, 3);
\draw [black,fill=smoothgreen] plot[mark=square*,mark size=5.4pt] (-2.598076, -4.5);
\draw [fill=lightgray] (-0.8660254, -4.5) circle (5.4pt);
\draw [fill=lightgray] (0.8660254, -4.5) circle (5.4pt);
\draw [black,fill=smoothred,rotate around={45:(-2.598076, 4.5)}] plot[mark=square*,mark size=5.4pt] (-2.598076, 4.5);
\draw [black,fill=smoothgreen] plot[mark=square*,mark size=5.4pt] (-0.8660254, 4.5);
\draw [fill=lightgray] (0.8660254, 4.5) circle (5.4pt);
\draw [black,fill=smoothred,rotate around={45:(-6.062178, -0.5)}] plot[mark=square*,mark size=5.4pt] (-6.062178, -0.5);
\draw [black,fill=smoothred,rotate around={45:(-5.196152, 1)}] plot[mark=square*,mark size=5.4pt] (-5.196152, 1);
\draw [black,fill=smoothred,rotate around={45:(-5.196152, -2)}] plot[mark=square*,mark size=5.4pt] (-5.196152, -2);
\draw [black,fill=smoothred,rotate around={45:(-4.330127, 2.5)}] plot[mark=square*,mark size=5.4pt] (-4.330127, 2.5);
\draw [black,fill=smoothred,rotate around={45:(-4.330127, -3.5)}] plot[mark=square*,mark size=5.4pt] (-4.330127, -3.5);
\draw [black,fill=smoothred,rotate around={45:(-3.464102, 4)}] plot[mark=square*,mark size=5.4pt] (-3.464102, 4);
\draw [fill=lightgray] (-3.464102, -5) circle (5.4pt);
\draw [fill=lightgray] (-1.732051, -5) circle (5.4pt);
\draw [fill=lightgray] (0, -5) circle (5.4pt);
\draw [black,fill=smoothgreen] plot[mark=square*,mark size=5.4pt] (-2.598076, 5.5);
\draw [fill=lightgray] (-0.8660254, 5.5) circle (5.4pt);
\draw [fill=lightgray] (0.8660254, 5.5) circle (5.4pt);
\draw [black,fill=smoothred,rotate around={45:(-6.928203, 0)}] plot[mark=square*,mark size=5.4pt] (-6.928203, 0);
\draw [black,fill=smoothred,rotate around={45:(-6.062178, -1.5)}] plot[mark=square*,mark size=5.4pt] (-6.062178, -1.5);
\draw [black,fill=smoothred,rotate around={45:(-6.062178, 1.5)}] plot[mark=square*,mark size=5.4pt] (-6.062178, 1.5);
\draw [black,fill=smoothred,rotate around={45:(-5.196152, -3)}] plot[mark=square*,mark size=5.4pt] (-5.196152, -3);
\draw [black,fill=smoothred,rotate around={45:(-5.196152, 3)}] plot[mark=square*,mark size=5.4pt] (-5.196152, 3);
\draw [black,fill=smoothgreen] plot[mark=square*,mark size=5.4pt] (-4.330127, -4.5);
\draw [black,fill=smoothred,rotate around={45:(-4.330127, 4.5)}] plot[mark=square*,mark size=5.4pt] (-4.330127, 4.5);
\draw [fill=lightgray] (-3.464102, 6) circle (5.4pt);
\draw [fill=lightgray] (-1.732051, 6) circle (5.4pt);
\draw [fill=lightgray] (0, 6) circle (5.4pt);
\draw [black,fill=smoothred,rotate around={45:(-7.794229, -0.5)}] plot[mark=square*,mark size=5.4pt] (-7.794229, -0.5);
\draw [black,fill=smoothred,rotate around={45:(-6.928203, 1)}] plot[mark=square*,mark size=5.4pt] (-6.928203, 1);
\draw [black,fill=smoothred,rotate around={45:(-6.928203, -2)}] plot[mark=square*,mark size=5.4pt] (-6.928203, -2);
\draw [black,fill=smoothred,rotate around={45:(-6.062178, 2.5)}] plot[mark=square*,mark size=5.4pt] (-6.062178, 2.5);
\draw [black,fill=smoothred,rotate around={45:(-6.062178, -3.5)}] plot[mark=square*,mark size=5.4pt] (-6.062178, -3.5);
\draw [black,fill=smoothred,rotate around={45:(-5.196152, 4)}] plot[mark=square*,mark size=5.4pt] (-5.196152, 4);
\draw [fill=lightgray] (-5.196152, -5) circle (5.4pt);
\draw [black,fill=smoothgreen] plot[mark=square*,mark size=5.4pt] (-4.330127, 5.5);
\draw [fill=lightgray] (-7.794229, -1.5) circle (5.4pt);
\draw [fill=lightgray] (-7.794229, 1.5) circle (5.4pt);
\draw [fill=lightgray] (-6.928203, -3) circle (5.4pt);
\draw [fill=lightgray] (-6.928203, 3) circle (5.4pt);
\draw [fill=lightgray] (-6.062178, -4.5) circle (5.4pt);
\draw [fill=lightgray] (-6.062178, 4.5) circle (5.4pt);
\draw [fill=lightgray] (-5.196152, 6) circle (5.4pt);
\draw [fill=lightgray] (-7.794229, 2.5) circle (5.4pt);
\draw [fill=lightgray] (-7.794229, -3.5) circle (5.4pt);
\draw [fill=lightgray] (-6.928203, 4) circle (5.4pt);
\draw [fill=lightgray] (-6.928203, -5) circle (5.4pt);
\draw [fill=lightgray] (-6.062178, 5.5) circle (5.4pt);
\draw [fill=lightgray] (-7.794229, -4.5) circle (5.4pt);
\draw [fill=lightgray] (-7.794229, 4.5) circle (5.4pt);
\draw [fill=lightgray] (-6.928203, 6) circle (5.4pt);
\draw [fill=lightgray] (-7.794229, 5.5) circle (5.4pt);
\draw [red,fill opacity=0,line width=1.8pt] (0, 0) circle (14.4pt);
\end{tikzpicture}
\caption{Turn $2\tau + 9$.}
\label{Strip Begin}
\end{subfigure}
\caption{Grid after turns $2\tau - 1$, $2\tau + 1$, $2\tau + 5$, and $2\tau + 9$; $\tau = 5$ and $f$ is circled.}\label{figure building the strip}
\end{figure}
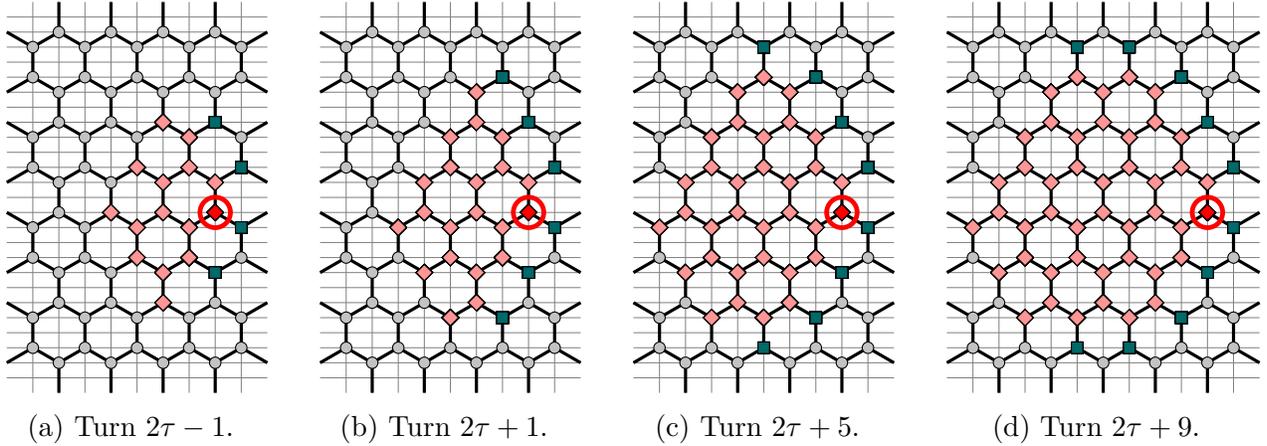

\subsection{Building a Protective Spiral}\label{spiral}
We will now bend the lower ray we built in Section \ref{protecc} into a clockwise spiral around the vertex $c = (-15\tau-13,0)$. Our goal is to construct this spiral in a way that it eventually collides with the upper ray, thus containing the fire. We first note where the actively burning vertices are.

\begin{observation}\label{observation activeburning}
After the fire spreads at turn $32\tau+27$, a vertex $(i,j)$ is actively burning if and only if $(i,j)$ is at distance $\tau$ from $c$ and $\frac{-3(\tau+1)}2 < j < 2+\frac{3(\tau+1)}2$ while $i<-15\tau-13$.
\end{observation}

\begin{proof}
By Lemma \ref{striplemma}, the vertices that are actively burning are the vertices on the line segments $(-16\tau-13,-1)$ to $(\frac{-31\tau-25}2,\frac{3\tau+1}2)$, and $(-16\tau-13,-1)$ to $(\frac{-31\tau-27}2,\frac{-3\tau+1}2)$. All of these vertices are at distance $\tau$ from $c$.
\end{proof}

When we build the protective spiral, we will do so in such that way that for every $s\geq 0$, on turn $2s+32\tau+28$ we protect a vertex $v_{s+16\tau+14}$ with $\dist(v_{s+16\tau+14},c)=\tau+s+1$. By Observation \ref{observation activeburning}, on turn $32\tau+28$ every actively burning vertex is at distance $\tau$ from $c$, so the placement of $v_{s+16\tau+14}$ will be a legal move. We start by noting that a shifted variant of Equation~\eqref{equation distance in hex grid} holds: Given a vertex $v=(i,j)$, the distance $\dist(c,v)=d$ in the hexagonal grid if and only if

\begin{equation}\label{equation distance to center spiral}
\max\left\{\frac{|2j - (d \bmod 2)|}{3},|i+15\tau+13|+\frac{|j+(d\bmod 2)|}{3}\right\} = d.
\end{equation}

The spiral is built by initially bending the lower ray $30\degree$ clockwise, and then bending it $60\degree$ clockwise at three later points in time. The initial $30\degree$ bend will occur at the vertex $(\frac{-31\tau-27}2,\frac{-3\tau-3}2)$, while the subsequent $60\degree$ bends will occur at the vertices $(-17\tau-15,0)$, $(-17\tau-15,6\tau+6)$, and $(-11\tau-9,12\tau+12)$. 

We first protect the vertices on the line segment $L_1$, which runs from $(\frac{-31\tau-27}2,\frac{-3\tau-3}2)$ to $(-17\tau-15,0)$, then we protect the vertices on the line segment $L_2$, from $(-17\tau-15,0)$ to $(-17\tau-15,6\tau+6)$, then the vertices on the line segment $L_3$, from $(-17\tau-15,6\tau+6)$ to $(-11\tau-9,12\tau+12)$, and finally, the vertices on the line segment $L_4$, from $(-11\tau-9,12\tau+12)$ to $(\frac{-\tau-3}2,\frac{3\tau+9}2)$. Note that if we were to extend this final line segment by one vertex, this vertex would be $(\frac{-\tau-1}2,\frac{3\tau+7}2)=(1-\frac{\tau-1}2,3+\frac{3(\tau+1)}2)$, which was protected in Section \ref{acc}. Hence, as long as the vertices along these line segments are indeed legal moves, the spiral has collided with the upper ray we built in sections \ref{acc} and \ref{protecc}, so we have successfully contained the fire. We now list the specific vertices which will be protected at each step so we can verify that they indeed are at the correct distances from $c$. Figure \ref{full figure} shows the end state we will reach after protecting the last vertex in $L_4$.

\begin{figure}[ht]
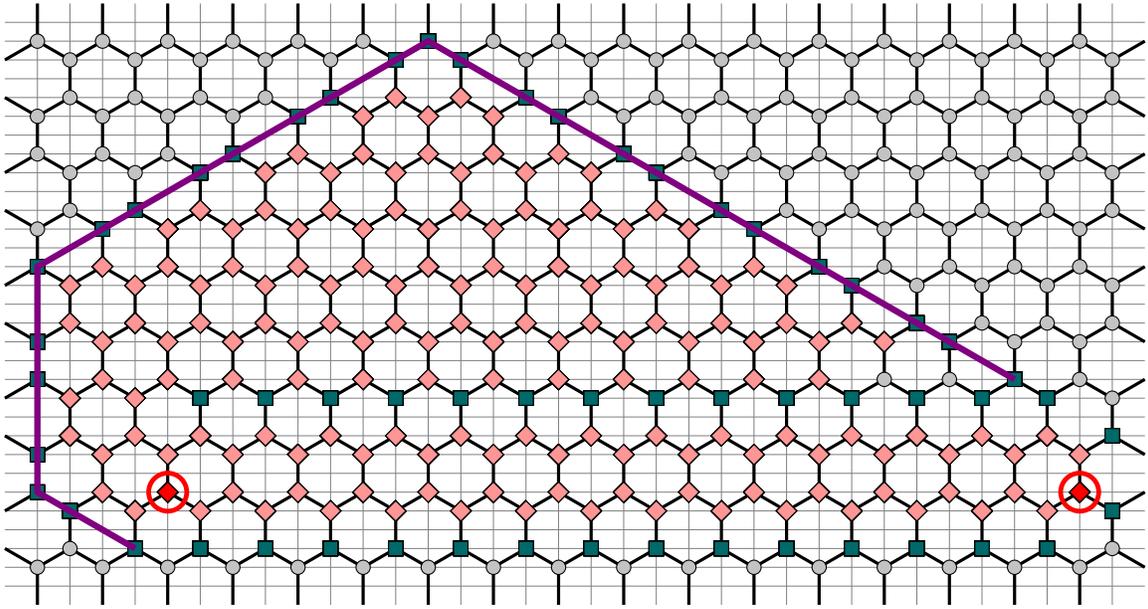

\centering

\caption{A complete picture of our strategy for $\tau = 1$, once the last vertex in $L_4$ has been protected. $f$ and $c$ are circled. The line segments $L_1,L_2,L_3$, and $L_4$ are drawn in violet.}
\label{full figure}
\end{figure}

Note that $L_1$ consists of $\tau+2$ vertices, so for $0\leq k\leq \frac{\tau-1}2$, on turn $4k+32\tau+28$, we will protect $v_{2k+16\tau+14}:=(\frac{-31\tau-27}2-3k,\frac{-3\tau-3}2+3k)$ and on turn $4k+32\tau+30$, we will protect $v_{2k+16\tau+15}:=(\frac{-31\tau-31}2-3k,\frac{-3\tau+1}2+3k)$. Then finally on turn $34\tau+30$, we protect $v_{17\tau+15}:=(-17\tau-15,0)$.

Then the line segment $L_2\setminus L_1$ consists of $2\tau+2$ vertices, so for each $0\leq k\leq \tau$, on turn $4k+34\tau+32$, we protect $v_{2k+17\tau+16}:=(-17\tau-15,6k+2)$, and on turn $4k+34\tau+34$, we protect $v_{2k+17\tau+17}:=(-17\tau-15,6k+6)$. This culminates when we protect $v_{19\tau+17}=(-17\tau-15,6\tau+6)$ on turn $38\tau+34$.

Now, the line segment $L_3\setminus L_2$ has $4\tau+4$ vertices, so for $0\leq k\leq 2\tau+1$, on turn $4k+38\tau+36$, we protect $v_{2k+19\tau+18}:=(-17\tau-13+3k,6\tau+8+3k)$ and on turn $4k+38\tau+38$, we protect $v_{2k+19\tau+19}:=(-17\tau-12+3k,6\tau+9+3k)$. We protect the last vertex on $L_3$ on turn $46\tau+42$ where we protect $v_{23\tau+21}=(-11\tau-9,12\tau+12)$

Finally, the line segment $L_4\setminus L_3$ has $7\tau+5$ vertices, so for $0\leq k\leq \frac{7\tau+3}2$, on turn $4k+46\tau+44$, we protect $v_{2k+23\tau+22}:=(-11\tau-8+3k,12\tau+11-3k)$, and on turn $4k+46\tau+46$, we protect $v_{2k+23\tau+23}:=(-11\tau-6+3k,12\tau+9-3k)$. This indeed finishes with vertex $v_{30\tau+26}=(\frac{-\tau-3}2,\frac{3\tau+9}2)$. Note that if we extended $L_4$ by one more vertex, we would arrive at $(\frac{-\tau-1}2,\frac{3\tau+7}2)=v_{\tau+2}$, so $L_4$ intersects the upper ray, containing the fire. See Figure \ref{Strat Outline} for a depiction of the line segments $L_1,L_2,L_3$ and $L_4$.

Now that we have described the remaining vertices we will protect, we will show that they indeed have the correct distance from $c$, implying that all these moves were legal moves.
 
\begin{lemma}\label{lemma distances from the center of the spiral}
The distance $\dist(c,v_{s+16\tau+14}) = \tau + s + 1$ for every $0\leq s\leq 14\tau+12$.
\end{lemma}

\begin{proof}
It suffices to verify Equation \eqref{equation distance to center spiral} for each point with the correct value of $d$. We will do so for the vertices on $L_1$ to show how this could be done, but omit the remaining calculations for brevity.

The vertices in $L_1$ correspond to $0\leq s\leq \tau+1$. When $s=2\ell$ for some $\ell$, the vertex $v_{s+16\tau+14}=(\frac{-31\tau-27}2-3\ell,\frac{-3\tau-3}2+3\ell)$, so
\begin{align*}
\max&\left\{\frac{|2((-3\tau-3)/2+3\ell) - 0|}{3},\left|\left(\frac{-31\tau-27}2-3\ell\right)+15\tau+13\right|+\frac{|((-3\tau-3)/2+3\ell)+0|}{3}\right\}\\
&=\max\left\{\tau+1-2\ell,\tau+1+2\ell\right\}=\tau+s+1,
\end{align*}
and when $s=2\ell+1$, the vertex $v_{s+16\tau+14}=(\frac{-31\tau-31}2-3\ell,\frac{-3\tau+1}2+3\ell)$, so
\begin{align*}
    \max&\left\{\frac{|2((-3\tau+1)/2+3\ell) - 1|}{3},\left|\left(\frac{-31\tau-31}2-3\ell\right)+15\tau+13\right|+\frac{|((-3\tau+1)/2+3\ell)+1|}{3}\right\}\\
    &=\max\left\{\tau-2\ell,\tau+2\ell+2\right\}= \tau+s+1.
\end{align*}
All other vertices in $L_2\cup L_3\cup L_4$ can be similarly verified.
\end{proof}

Lemma \ref{lemma distances from the center of the spiral} shows that each move in Section \ref{spiral} was legal, which completes the proof of Theorem \ref{theorem main theorem}.

\section{Birth Sequence Trees and Forests}
Here we study $k$-containability for birth sequence trees and forests. We will use the following definitions through this section. The operation $\odot$ refers to the component-wise multiplication of vectors. For example, $[2,3,5] \odot [3,5,7] = [6,15,35]$. The $\odot$ can be omitted: $\vec{x}\vec{y} = \vec{x} \odot \vec{y}$. And $\bigodot_{i=1}^n\vec{v_i} = \vec{v_i} \odot \vec{v_2} \odot \dots \odot \vec{v_n}$. All operations on vectors are component-wise. For example, let $\vec{x}$ and $\vec{y}$ be $m$-component vectors. $\max\{\vec{x},\vec{y}\}$ results in the $m$-component vector $\vec{z}$, such that the $j^{th}$ component of $\vec{z}$ is the maximum of the $j^{th}$ component of $\vec{x}$ and $\vec{y}$. And $\vec{x} \leq \vec{y}$ implies that for each $j$, the $j^{th}$ component of $\vec{x}$ is less than or equal to the $j^{th}$ component of $\vec{y}$. Further, we define $\bigodot_{i=1}^0\vec{x_i} = \vec{1}$, the vector with all components equal to $1$, and with as many components as $x_i$.

\subsection{Infinite Birth Sequence Trees}
The following two lemmas will simplify the analysis of firefighting on trees. The first lemma was originally observed by MacGillivray and Wang.

 \begin{lemma}[\cite{MW2003}]\label{lemma hotstrat}
If there exists a strategy that contains a fire on a forest $F$, then there exists a hot strategy that contains a fire on $F$. 
 \end{lemma}

The next lemma allows us to exploit the symmetry inherent in birth sequence trees.

\begin{lemma}\label{lemma birth sequence any hot strategy for trees}
When firefighting on a birth sequence tree with a burning root using a hot strategy, it does not matter which vertices are protected.
\end{lemma}

\begin{proof} Let $T$ be a birth sequence tree with a burning root, and assume we use a hot strategy. Let $t \in \mathbb{N}$. On turn $2t$, The fire has spread $t - 1$ times, so any vertex $v$ at depth less than $t$ is on fire, protected, or saved by a protected vertex blocking the path between $v$ and the root. Regardless of which vertices we protect on turn $2t$, after the fire spreads on turn $2t+1$, the same number of isomorphic subtrees will be left unprotected. Thus, it does not matter which vulnerable vertices are protected.
\end{proof}


Next, we prove Theorem \ref{theorem infinite trees}, which is a general necessary and sufficient condition for $k$-containability for infinite birth sequence trees. 

\begin{proof}[Proof of Theorem \ref{theorem infinite trees}]
Let $T$ be an infinite rooted tree generated by the birth sequence $d_0,d_1,d_2,\dots$, with its root initially on fire. Via Lemma \ref{lemma hotstrat}, it suffices to only consider hot strategies. By Lemma \ref{lemma birth sequence any hot strategy for trees}, it does not matter which vertices we protect as long as we follow a hot strategy. If the number of vertices on fire at depth $t$ is $f_t$, then the number of vertices that will be on fire at depth $t + 1$ is given by $\max\{0,(f_t\cdot d_t) - k\}$. This is because each of the $f_t$ burning vertices will have $d_t$ children, all of which are vulnerable, and by protecting at most $k$ of these, we will end up with $(f_t\cdot d_t) - k$ or $0$ burning vertices at depth $t + 1$, whichever is less. This gives a recursive sequence: $f_0 = 1$, and for all $t \in \mathbb{N}$, $f_t = \max\{0,(f_{t - 1}\cdot d_t) - k\}$, which has solution
\[
f_t = \max\left\{0,\prod_{i=0}^{t-1}d_i - k\left(1+\sum_{i=1}^{t-1}\prod_{j=i}^{t-1}d_j\right)\right\}.
\] 
If there exists $t \in \mathbb{Z}_{\geq0}$ such that 
\[
\prod_{i=0}^td_i - k\left(1+\sum_{i=1}^t\prod_{j=i}^td_j\right)\leq 0,
\]
then $f_{t+1} = 0$, which means the fire is contained at depth $t + 1$. Otherwise, for all $t \in \mathbb{N}$, we have $f_t \geq 1$, which implies that the fire spreads indefinitely.
\end{proof}

\subsection{Infinite Birth Sequence Forests}
Here we consider forests made up of finitely many infinite birth sequence trees. Suppose we have $m$ infinite birth sequence trees, labeled $1$ through $m$, each with its root on fire. Let $\vec{d}_0,\vec{d}_1,\vec{d}_2,\dots$ be an infinite list of $m$-component vectors such that the $j^{th}$ component of $\vec{d}_i$ is the number of children that vertices of Tree $j$ at depth $i$ have. Lemma \ref{lemma hotstrat} still applies in this case, but Lemma \ref{lemma birth sequence any hot strategy for trees} does not, so we may need to consider different hot strategies. A hot strategy in this case can be expressed as a sequence of $m$-component vectors of non-negative integers $\vec{p}_1,\vec{p}_2,\dots$ such that the $j^{th}$ component of $\vec{p}_i$ indicates the number of firefighters placed on Tree $j$ at depth $i$. If we are considering $k$-containability, then we have $||\vec{p}_i||_1 \leq k$ for all $i$ since we use at most $k$ firefighters per turn. We will call the sequence $\vec{p}_1,\vec{p}_2,\dots$ a \textbf{strategy sequence}.


\begin{theorem}\label{theorem infinite forests}
Let $F$ be the forest of infinite rooted trees generated by the birth sequence vectors $\vec{d}_0,\vec{d}_1,\vec{d}_2,\dots$, where each tree has its root on fire. Then $F$ is $k$-containable if and only if there exists some $t \in \mathbb{Z}_{\geq 0}$ and non-negative integer vectors $\vec{p}_1,\vec{p}_2,\dots,\vec{p}_{t+1}$ with $||\vec{p}_i||_1 = k$ for all $i \in [t+1]$ such that 
\[
\bigodot_{i=0}^t\vec{d}_i - \left(\sum_{i=1}^{t} \vec{p}_i \bigodot_{j=i}^t \vec{d}_j \right) - \vec{p}_{t+1} \leq \vec{0}.
\]
\end{theorem}

\begin{proof}
The proof of this theorem is similar to the proof of Theorem \ref{theorem infinite trees}. Let $\vec{p}_1,\vec{p}_2,\dots$ be a strategy sequence and let $\vec{f}_t$ be the $m$-component vector with the $j^{th}$ component of $\vec{f}_t$ representing the number of fires that Tree $j$ has at depth $t$. Note that $\vec{f}_t$ is a recursive sequence: $\vec{f}_0 = \vec{1}$, and for all $t \in \mathbb{N}$, 
\[
\vec{f}_t = \max\{\vec{0},(\vec{f}_{t - 1} \odot \vec{d}_{t - 1}) - \vec{p}_{t}\}.
\]
We will show via induction that for all $t \in \mathbb{N}$, 
\begin{equation}\label{equation forests induction}
\vec{f}_t = \max\left\{\vec{0},\bigodot_{i=0}^{t-1}\vec{d}_i - \left(\sum_{i=1}^{t-1} \vec{p}_i \bigodot_{j=i}^{t-1} \vec{d}_j \right) - \vec{p}_{t}\right\}.
\end{equation}
Note that 
\[
\vec{f}_1 = \max\{\vec{0},(\vec{f}_{0} \odot \vec{d}_{0}) - \vec{p}_{1}\} = \max\{\vec{0},\bigodot_{i=0}^{0}\vec{d}_i - \left(\sum_{i=1}^{0} \vec{p}_i \bigodot_{j=i}^{0} \vec{d}_j \right) - \vec{p}_{1}\}.
\]
Now, assume that Equation \eqref{equation forests induction} holds for some $t \in \mathbb{N}$. Then we have 
\begin{align*}
\vec{f}_{t+1} &= \max\{\vec{0},(\vec{f}_t \odot \vec{d}_t) - \vec{p}_{t+1}\}\\
&=\max\left\{\vec{0},\left(\max\left\{\vec{0},\bigodot_{i=0}^{t-1}\vec{d}_i - \left(\sum_{i=1}^{t-1} \vec{p}_i \bigodot_{j=i}^{t-1} \vec{d}_j \right) - \vec{p}_{t}\right\} \odot \vec{d}_t\right) - \vec{p}_{t+1}\right\}\\
&=\max\left\{\vec{0},\max\left\{\vec{0} \odot \vec{d}_t - \vec{p}_{t+1},\left(\bigodot_{i=0}^{t-1}\vec{d}_i - \left(\sum_{i=1}^{t-1} \vec{p}_i \bigodot_{j=i}^{t-1} \vec{d}_j \right) - \vec{p}_{t}\right)\vec{d}_t - \vec{p}_{t+1}\right\}\right\}\\
&=\max\left\{\vec{0},\max\left\{-\vec{p}_{t+1},\bigodot_{i=0}^{t}\vec{d}_i - \vec{d}_t\left(\sum_{i=1}^{t-1} \vec{p}_i \bigodot_{j=i}^{t-1} \vec{d}_j \right) - \vec{d}_t\vec{p}_{t}-\vec{p}_{t+1}\right\}\right\}\\
&=\max\left\{\vec{0},\bigodot_{i=0}^{t}\vec{d}_i - \left(\sum_{i=1}^{t-1} \vec{p}_i \bigodot_{j=i}^{t} \vec{d}_j \right) - \vec{d}_t\vec{p}_{t}-\vec{p}_{t+1}\right\}\\
&=\max\left\{\vec{0},\bigodot_{i=0}^{t}\vec{d}_i - \left(\sum_{i=1}^{t} \vec{p}_i \bigodot_{j=i}^{t} \vec{d}_j \right) - \vec{p}_{t+1}\right\}.
\end{align*}
If there exists $t \in \mathbb{Z}_{\geq0}$ such that $\bigodot_{i=0}^t\vec{d}_i - \left(\sum_{i=1}^{t} \vec{p}_i \bigodot_{j=i}^t \vec{d}_j \right) - \vec{p}_{t+1} \leq \vec{0}$, then $\vec{f}_{t+1} = \vec{0}$, which means the fire is contained at depth $t + 1$, and thus $F$ is $k$-containable. On the other hand, if for all strategy sequences there is no such $t$, then the fire spreads forever, so $F$ is not $k$-containable.
\end{proof}

Note that Theorem \ref{theorem infinite trees} is a special case of Theorem \ref{theorem infinite forests}, where the vectors are all one-dimensional, and therefore all strategy vectors $\vec{p}_i$ are exactly $[k]$.

\subsection{Finite Birth Sequence Forests}
Theorem \ref{theorem infinite forests} can also be used to imply results about finite forests. Here we consider forests of $m$ finite birth sequence trees with depth $n$, generated by birth sequence vectors $\vec{d}_0,\vec{d}_1,\dots,\vec{d}_{n-1}$. Since the question of $k$-containability is trivial for finite graphs, in this section instead we will consider the question of if $k$ firefighters can stop the spread of the fire before any leafs are on fire. The answer to this question is positive if and only if there exists non-negative integer vectors $\vec{p}_1,\vec{p}_2,\dots,\vec{p}_n$ with $L_1$ norm $k$ such that we have 
\[
\bigodot_{i=0}^{n-1} \vec{d}_i - \bigg(\sum_{i=1}^{n-1}\vec{p}_i\bigodot_{j=i}^{n-1}\vec{d}_j\bigg) - \vec{p}_n \leq \vec{0}.
\]
However, checking this condition by brute force is inefficient, as the number of possible vectors $\vec{p}_1,\vec{p}_2,\dots,\vec{p}_n$ grows exponentially with $n$. We present an efficient pseudopolynomial algorithm for fixed $m$, resulting from the following lemma.

\begin{lemma}\label{lemma too many trees reduction}
Suppose we are firefighting on a forest of $m$ finite birth sequence trees with depth $n$, using $k$ firefighters per turn. If at some point, the number of fires in the forest at depth $t$ exceeds $k(n-t)$, then it is impossible to save all leaves.
\end{lemma}

\begin{proof}
On any turn, at most $k$ trees can have firefighters placed on them, so if there is a point in time where there are more than $k(n-t)$ subtrees with actively burning roots and only $n-t$ even turns left before the fire reaches the leaves, one of the subtrees is guaranteed to have its leaves burnt.
\end{proof}

We will construct a state graph which will be helpful in determining if a successful strategy sequence exists. As we are firefighting with a hot strategy, we can describe the state we are at before placing a firefighter with a tuple $(t,\vec{f})$, where $t\in \{0,1,\dots,n\}$ and $\vec{f}$ is an $m$-component non-negative integer vector. At the state $(t,\vec{f})$, the fire has spread $t$ times, and $f_j$ is the number of fires in tree $j$ at depth $t$. The beginning state is $(0,\vec{1})$. We can successfully contain the fire if and only if we can reach the end state $(n,\vec{0})$. By Lemma \ref{lemma too many trees reduction}, any state $(t,\vec{f})$ such that $||\vec{f}||_1 > k(n-t)$ cannot lead to the successful end state, and thus such states can be omitted from our state graph. We construct a state graph $G=(V,E)$ of all viable states by following Algorithm \ref{algorithm state graph}.

{
\begin{algorithm2e}[htp!]
\caption{Constructing The State Graph}\label{algorithm state graph}
\SetKw{Continue}{continue}
\SetKw{And}{and}
$P \gets \{\vec{v}\in (\mathbb{Z}_{\geq 0})^m\ :\ ||\vec{v}||_1 = k\}$\;
$V \gets \{(0,\vec{1})\} \cup \{(t,\vec{f}) \in \{1,\dots,n\}\times (\mathbb{Z}_{\geq 0})^m\ :\ ||\vec{f}||_1 \leq k(n-t)\}$\;
$E \gets \varnothing$\;
\For{$(t,\vec{f}) \in V \setminus \{(n,\vec{0})\}$}{
    \If {$||\vec{d}_t \odot \vec{f}||_1 \leq k$}{
        add $(t,\vec{f})\rightarrow(t+1,\vec{0})$ to $E$\;
    }\Else {
        \For{$p \in P$}{
            $\vec{g} \gets (\vec{d}_t \odot \vec{f}) - p$\;
            \If{$\vec{g} \geq \vec{0}$ \And $||\vec{g}||_1 \leq k(n-t-1)$}{
                add $(t,\vec{f})\rightarrow(t+1,\vec{g})$ to $E$\;
            }
        }
    }
}
$G \gets (V,E)$\;
\end{algorithm2e}
}
We can use $G$ to check if a successful strategy exists by checking whether $(n,\vec{0})$ is reachable from $(0,\vec{1})$ in $G$. Doing this takes $O(|V| + |E|)$ time, and $G$ can be constructed in $O(m(|V| + |E|))$ time (since states are expressed by $m + 1$ integers), our algorithm to check whether a successful strategy exists takes $O(m(|V| + |E|))$ time. We have $V \subseteq \{0,1,\dots,n\}\times\{0,1,\dots,kn\}^m$, so $|V| \leq n(kn)^m$, and we have $|E| \leq |P||V| = \binom{m + k - 1}{m - 1}|V|$, so our algorithm runs in $O(m\binom{m + k - 1}{m - 1}n(kn)^m)$ time. For fixed $m$, this time complexity is polynomial in the values of $n$ and $k$, so our algorithm runs in pseudopolynomial time, since it is exponential in the number of bits needed to specify $k$.

Lastly, note that while this algorithm runs in exponential time if $m$ is allowed to vary, we only need to consider cases where $m \leq nk$. This is because otherwise, we instantly know that saving all leaves is impossible, even if all $m$ birth sequence trees are paths.

\section{Acknowledgements}
The authors would like to thank the Illinois Geometry Lab for facilitating this research project. This material is based upon work supported by the National Science Foundation under Grant No. DMS-1449269. Any opinions, findings, and conclusions or recommendations expressed in this material are those of the authors and do not necessarily reflect the views of the National Science Foundation.

\bibliographystyle{abbrv}
\bibliography{refs}

\newpage
\subsection*{Appendix: An Algorithm for Containment on the Hexagonal Grid}\vspace{-.2cm}
Below, we provide a full implementation of our strategy for containing the fire in the hexagonal grid with one extra firefighter in code. Our algorithm uses the extra firefighter on turn $2t$, where $t$ is odd.

{
\begin{algorithm2e}[H]
\caption{Containing the Fire with One Extra Firefighter}
\tcp{Section \ref{protecc2/3}}
\For{$k \gets 0$ \KwTo $\frac{t - 3}{2}$}{
    protect $(1-k, -1 - 3k)$\;
    protect $(1-k, 3 + 3k)$\;
}
\tcp{Section \ref{acc} -- Next turn we use the extra firefighter.}
$k \gets \frac{t - 1}{2}$\;
protect $(1-k, -1 - 3k)$ and $(1-k, 3 + 3k)$\;
\tcp{Section \ref{protecc}}
$k \gets \frac{t+1}{2}$\;
\For{$j \gets 0$ \KwTo $15k-2$}{
    protect $(-k-2j, -3k)$\;
    protect $(-k-2j, 2 + 3k)$\;
}
\tcp{Section \ref{spiral}}
$(i,j) \gets (-31k+2,-3k)$\;
\tcp{$(i,j)$ is the vertex we block on the spiral as we build it.}
protect $(i,j)$\;
\For{$n \gets 1$ \KwTo $2k$}{
    $m \gets 2$ if $n$ is odd, and $1$ otherwise\;
    $i \gets i - m$\;
    $j \gets j + m$\;
    protect $(i,j)$\;
}
\For{$n \gets 1$ \KwTo $4k$}{
    $m \gets 1$ if $n$ is odd, and $2$ otherwise\;
    $j \gets j + 2m$\;
    protect $(i,j)$\;
}
\For{$n \gets 1$ \KwTo $8k$}{
    $m \gets 2$ if $n$ is odd, and $1$ otherwise\;
    $i \gets i + m$\;
    $j \gets j + m$\;
    protect $(i,j)$\;
}
\For{$n \gets 1$ to $14k - 2$}{
    $m \gets 1$ if $n$ is odd, and $2$ otherwise\;
    $i \gets i + m$\;
    $j \gets j - m$\;
    protect $(i,j)$\;
}
\end{algorithm2e}
}

\end{document}